\documentclass[pdflatex,sn-mathphys-num]{sn-jnl}

\usepackage{graphicx}%
\usepackage{multirow}%
\usepackage{amsmath,amssymb,amsfonts}%
\usepackage{amsthm}%
\usepackage{mathrsfs}%
\usepackage[title]{appendix}%
\usepackage{xcolor}%
\usepackage{textcomp}%
\usepackage{manyfoot}%
\usepackage{booktabs}%
\usepackage{algorithm}%
\usepackage{algorithmicx}%
\usepackage{algpseudocode}%
\usepackage{listings}%

\newtheorem{theorem}{Theorem}[section]
\newtheorem{lemma}[theorem]{Lemma}
\newtheorem{corollary}[theorem]{Corollary}
\newtheorem{proposition}[theorem]{Proposition}

\theoremstyle{definition}
\newtheorem{definition}[theorem]{Definition}
\newtheorem{example}[theorem]{Example}

\theoremstyle{remark}
\newtheorem{remark}[theorem]{Remark}

\begin{document}

\title[Article Title]{Cowen-Douglas operators on quaternionic Hilbert spaces}


\author[1]{\fnm{Xiaoqi} \sur{Feng}}\email{fengxiaoqi2011@qq.com}

\author*[1]{\fnm{Bingzhe} \sur{Hou}}\email{houbz@jlu.edu.cn}
\equalcont{These authors contributed equally to this work.}

\author[2]{\fnm{Kui} \sur{Ji}}\email{jikui@hebtu.edu.cn}
\equalcont{These authors contributed equally to this work.}

\affil[1]{\orgdiv{School of Mathematics}, \orgname{Jilin University}, \orgaddress{\street{Qianjin Street}, \city{Changchun}, \postcode{130012}, \country{China}}}

\affil[2]{\orgdiv{School of Mathematics}, \orgname{Hebei Normal University}, \orgaddress{\street{Nan'erhuan East Road}, \city{Shijiazhuang}, \postcode{050016}, \country{China}}}


\abstract{In 1978, M. J. Cowen and R. G. Douglas introduced a class of geometric operators  (known as Cowen-Douglas class of operators) and associated a Hermitian holomorphic vector bundle to such operators. They gave a complete set of unitary invariants in terms of the curves and its covariant derivatives. In this paper, after giving some basic properties of $S$-spectrum and right eigenvalues of bounded right linear operators on separable quaternionic Hilbert spaces, we generalize the class of Cowen-Douglas operators to the quaternionic Hilbert space via the $S$-spectrum and denote this class as $B_n^s(\Omega_q)$. Due to the lack of commutativity of quaternion multiplication, the quaternionic Cowen-Douglas operators are not trivial generalizations of the classical Cowen-Douglas operators. Each operator in $B_n^{s}(\Omega_q)$ corresponds to an $n$-dimensional Hermitian right holomorphic quaternionic vector bundle. We first establish a rigidity theorem for Hermitian right holomorphic quaternionic vector bundles. It is then proven that two operators in $B_n^{s}(\Omega_q)$ are quaternion unitarily equivalent if and only if the associate bundles are equivalent as Hermitian right holomorphic quaternionic vector bundles. In particular, we introduce canonical matrix representations of operators in $B_1^{s}(\Omega_q)$ and furthermore, we give the quaternion unitarily equivalent classification of $B_1^{s}(\Omega_q)$ by the canonical matrix representations. It is worth noting that curvature is a complete unitary invariant for the classical (complex) Cowen-Douglas operators, however, there exist two quaternionic Cowen-Douglas operators which have the same curvature but are not quaternion unitarily equivalent. In addition, we prove that the operators in $B_1^{s}(\Omega_q)$ are quaternion unitarily equivalent if and only if their complex representations are unitarily equivalent. Some relevant examples of the above results are also provided.}




\maketitle

\section{Introduction}\label{sec1}

The quaternions form a unital non-commutative division algebra and play an important role in quantum physics \cite{AF, FJSS, E, HB}. From the
mathematical point of view, one of the most important investigations in the quaternionic context is to find a convenient manner to express the
``analyticity" of functions with respect to quaternions. In 2007, G. Gentili and C. Struppa \cite{GS} introduced a concept of slice regularity for functions of one quaternionic variable, which led to a large development \cite{CSS, GP}. Another important investigation is the unitary representation of groups on quaternionic
Hilbert spaces \cite{FJS, NV}. In the case of finite dimension, quaternionic matrices have been widely studied \cite{L, W} (see F. Zhang's survey \cite{Z} for example), especially the right eigenvalue problems \cite{LS, LSS}. In the case of infinite dimension, the normal operators and unitary
operators on quaternionic separable Hilbert spaces were well studied, one can see \cite{CCKS, RK, SC, V} for more details.
In recent years, F. Colombo et al. developed a new theory of quaternionic functional calculus by  $S$-spectrum and $S$-resolvent set for quaternionic linear operators \cite{CSS}.

In 1978, M. Cowen and R. Douglas \cite{CD} introduced a class of geometric operators $B_n(\Omega)$ on a separable complex Hilbert space $\mathcal {H}$, known as Cowen-Douglas class of operators.

\begin{definition}[\cite{CD}]
	For $\Omega$ a connected open subset of $\mathbb{C}$ and $n$ a
	positive integer, let ${B}_{n}(\Omega)$ denotes the
	operators $T$ in $\mathcal {L}(\mathcal {H})$  satisfying:
	\begin{enumerate}
		\item[(a)] \ $\Omega \subseteq \sigma(T)=\{\omega \in \mathbb{C}: T-\omega \
	\text{not  invertible}\}$;
	
\item[(b)] \ ${\rm Ran}(T-\omega)=\mathcal {H} \ \text{for} \ \omega \ \text{in} \ \Omega$;
	
\item[(c)] \ $\bigvee {\rm Ker}_{\omega\in \Omega}(T-\omega)=\mathcal {H}$;
	
	\item[(d)] \ ${\rm dim \ Ker}(T-\omega)=n$ for $\omega$ in $\Omega$.
		\end{enumerate}
\end{definition}
It is well known that study of Cowen-Douglas operators is closely connected with  complex Hermitian geometry. In \cite{CD}, M. Cowen and R. Douglas investigated that
$B_{n}(\Omega)$ corresponds to holomorphic curves on a complex Hilbert space $\mathcal {H}$, which naturally induce $n$-dimensional Hermitian holomorphic vector bundles, they also
proved that the unitary classification of $B_n(\Omega)$ is related to the isometrically isomorphic classification of holomorphic curves.

Moreover, C. Jiang and K. Ji considered
the similarity classification of Cowen-Douglas operators in \cite{JGJ, JJ}. G. Misra \cite{M} studied the Cowen-Douglas operator class and its related topics. In recent years,  the
similarity invariant of Cowen-Douglas operators has been considered, one can see \cite{JJK, JJW}.  Furthermore, K. Ji et al considered the tuples of Cowen-Douglas operators \cite{HJJX, JJKX}.

In this article, the basic idea is to introduce an analogue of Cowen-Douglas operator on a separable quaternionic Hilbert spaces.  Denote by $\mathcal{H}_q$ a separable quaternionic Hilbert space and denote by $L_r(\mathcal{H}_q)$ the set of all the bounded right linear quaternionic operators $\mathcal{H}_q$.

\begin{definition}[\cite{CSS}]\label{SSP}
	Let $T\in L_r(\mathcal{H}_q)$. The set of all $S$-spectrum of $T$ is defined by
	\[
	\sigma_s(T)=\{s\in\mathbb{H};\ T^2-2\mathrm{Re}(s)T+|s|^2\boldsymbol{I}~\text{is not invertible}\}.
	\]
    where $I$ means the identity operator.
\end{definition}

For a subset $F\subseteq \mathcal{H}_q$,  we define
\[
\rm{span}_{\mathbb{H}}F:=\{\sum\limits_{\mathbf{a}\in A} \mathbf{a}q_{\mathbf{a}}; A\subseteq F \ \text{is a finite subset and} \ q_{\mathbf{a}}\in \mathbb{H}\},
\]
and
\begin{equation*}
	{\bigvee}_{\mathbb{H}}F:=\overline{\rm{span}_{\mathbb{H}}F}.
\end{equation*}

\begin{definition}\label{SPCD}
	For $\Omega_q$ a connected open subset  of $\mathbb{H}$ and $n$ a positive integer, let $B_n^{s}(\Omega_q)$ denote the operators in $L_r(\mathcal{H}_q)$  satisfy:
\begin{enumerate}
	\item[(a)] \ $\Omega_q\subseteq\sigma_{s}(T)$; \\
	\item[(b)] \ $\mathrm{ran}(T^2-2\mathrm{Re}(\omega) T+\vert\omega \vert^2)=\mathcal{H}_q$, for any $\omega\in\Omega_q$;\\
	\item[(c)] \ $\bigvee_{\mathbb{H}} \{\ker(T^2-2\mathrm{Re}(\omega) T+\vert\omega \vert^2);\ \omega\in\Omega_q\}=\mathcal{H}_q$;\\
	\item[(d)] \ $\dim_{\mathbb{H}} \ker(T^2-2\mathrm{Re}(\omega)T+\vert\omega \vert^2)=n$, for any $\omega\in\Omega_q$.
\end{enumerate}
\end{definition}

In order to study the generalization of Cowen-Douglas operators on quaternionc Hilbert spaces, we will investigate both geometry and operator theory on quaternionic Hilbert spaces.

In Section 2, we give some related notations and definitions of bounded right linear operators on quaternionic Hilbert spaces, and introduce the quaternion-valued right holomorphic functions on a domain (connected open subset) in $\mathbb{C}$.

In Section 3, we give the basic properties of $S$-point spectrum and right eigenvalues of bounded right linear operators on quaternionic Hilbert spaces. Then, we study some fundamental properties of Cowen-Douglas operators induced by $S$-point spectrum. We also show the backward unilateral weighted shift operators in the class of quaternionic Cowen-Douglas operators.

In Section 4, We give a rigidity theorem for Hermitian right holomorphic quaternionic vector bundles. Based on this theorem we further prove that two operators in $B_n^{s}(\Omega_q)$ are unitarily equivalent if and only if the associate bundles are equivalent as Hermitian right holomorphic quaternionic vector bundles.

In the last section, we introduce canonical matrix representations of operators in $B_1^{s}(\Omega_q)$. Based on the canonical matrix representations, we give the unitarily equivalent classification of $B_1^{s}(\Omega_q)$. In particular, we provide an example to show that curvature is not a complete unitary invariant,  which is different from the complex linear operators in $B_1(\Omega)$. In addition, we show that two  operators in $B_1^{s}(\Omega_q)$ are unitarily equivalent if and only if their complex representations are unitarily equivalent.

\section{Preliminaries}\label{sec2}

In this section, let us review some basic facts of quaterninonic numbers, quaterninonic Hilbert spaces, quaterninonic linear operators and quaterninon-valued holomorphic functions.

Denote the set of all quaternions by $\mathbb{H}$. Each element $a\in\mathbb{H}$ has the form
\[
a=a_0+a_1\boldsymbol{i}+a_2\boldsymbol{j}+a_3\boldsymbol{k},\ \ \ a_0, a_1, a_2, a_3\in\mathbb{R},
\]
where $\boldsymbol{i}^2=\boldsymbol{j}^2=\boldsymbol{k}^2=-1$, $\boldsymbol{i}\boldsymbol{j}=-\boldsymbol{j}\boldsymbol{i}=\boldsymbol{k}$,
$\boldsymbol{j}\boldsymbol{k}=-\boldsymbol{k}\boldsymbol{j}=\boldsymbol{i}$ and
$\boldsymbol{k}\boldsymbol{i}=-\boldsymbol{i}\boldsymbol{k}=\boldsymbol{j}$. Denote the real and image part of $a$  by $\mathrm{Re}(a)=a_0$ and $\mathrm{Im}(a)=a_1\boldsymbol{i}+a_2\boldsymbol{j}+a_3\boldsymbol{k}$, respectively. The modulus of $a$ is defined by
$|a|=\sqrt{a_0^2+a_1^2+a_2^2+a_3^2}$, and the conjugate of $a$ is defined by $\overline{a}=a_0-a_1\boldsymbol{i}-a_2\boldsymbol{j}-a_3\boldsymbol{k}$.

Notice that $\{1, \boldsymbol{i}, \boldsymbol{j}, \boldsymbol{k}\}$ is the standard basis of $\mathbb{H}$.
Denote by $\mathbb{S}$ the unit sphere of purely imaginary quaternions, i.e.,
\[
\mathbb{S}=\{a=a_1\boldsymbol{i}+a_2\boldsymbol{j}+a_3\boldsymbol{k}\in\mathbb{H};\ a_1^2+a_2^2+a_3^2=1\}.
\]
For any  $q\in\mathbb{H}\setminus\mathbb{R}$, define $I_q=\frac{\mathrm{Im}(q)}{|\mathrm{Im}(q)|}$, then $I_q\in\mathbb{S}$. Furthermore,
denote
\[
\mathbb{C}_{I_q}=\{a+I_qb; a,b\in \mathbb{R}\}.
\]
Then for any  $q\in\mathbb{H}\setminus\mathbb{R}$, $\mathbb{C}_{I_q}$ is isomorphic to $\mathbb{C}$ and $\mathbb{C}_{I_q}$ is just the multiplication commutant of $q$ in $\mathbb{H}$, i.e.,
\[
\mathbb{C}_{I_q}=\{p\in\mathbb{H}; pq=qp\}.
\]

For $p,\ q\in\mathbb{H}$, we say that $p$ and $q$ are axially symmetric,  denoted by $p\sim q$, if $|p|=|q|$ and $\mathrm{Re}(p)=\mathrm{Re}(q)$. Following from \cite{CGSS, HT}, one can see that $p\sim q$ if and only if there exists $0\neq \omega\in\mathbb{H}$ such that $q=\omega^{-1}p\omega$, or equivalently, there exists $\omega\in\mathbb{H}$ with $|\omega|=1$ such that $q=\overline{\omega}p\omega$. Moreover,
A subset $\Omega_q\subseteq\mathbb{H}$ is said to be an axially symmetric set if for every $x\in\Omega_q$, the set $\{y\in\mathbb{H};\ y\sim x\}$ is contained in $\Omega_q$.
Furthermore, we introduce the concept of axially symmetric reduced set for axially symmetric sets.

\begin{definition}
    Let $\Omega_q$ be a subset in $\mathbb{H}$. For any $\omega\in\Omega_q$,
    let $\omega_{red}=\mathrm{Re}(\omega)+\boldsymbol{i}|\mathrm{Im}(\omega)|$. Denoted by $\Omega_{red}$ the set of all $\omega_{red}$. We call $\Omega_{red}$ the axially symmetric reduced subset of $\Omega_q$.
\end{definition}
Throughout this paper, $\Omega$ is used to denote an open connected set in $\mathbb{C}$, $\Omega_q$ is used to denote an axially symmetric set of $\mathbb{H}$ and $\Omega_{red}$ is used to denote the axially symmetric reduced set of $\Omega_q$.

Now, let us review the quaternionic Hilbert spaces and quaternionic linear operators (we refer to \cite{CSS}). Because of the noncommutative multiplication of quaternions, the vector spaces (Hilbert spaces) over quaternion field $\mathbb{H}$ are sightly different from the vector spaces over complex field $\mathbb{C}$.

Since the noncommutative multiplications of quaternions, the results is signification differences between vector spaces over the quaternion field
$\mathbb{H}$ and those over the complex field. We therefore briefly recall some definitions and properties related to quaternionic Hilbert spaces and quaternionic linear operators (we refer to \cite{CSS}).

\begin{definition}
Let $X$ be a nonempty set. $X$ is called a right(left) linear quaternionic vector space if for any $x,\ y,\ z\in X$ and any $\alpha,\ \beta\in\mathbb{H}$, the following conditions hold.
\[
\begin{aligned}
(1) \ & x+y=y+x, \\
(2) \ & x+(y+z)=(x+y)+z, \\
(3) \ & \text{there is a unique zero element}\ 0\ \text{satisfying}\ x+0=x, \text{for any}\ x\in X,\\
(4) \ & \text{for every}\ x,\ \text{there is a unique element}\ -x\ \text{satisfying}\ x+(-x)=0, \\
(5) \ & (x+y)\alpha=x\alpha+y\alpha, (\alpha(x+y)=\alpha x+\alpha y) \\
(6) \ & x(\alpha+\beta)=x\alpha+x\beta, ((\alpha+\beta) x=\alpha x+\beta x) \\
(7) \ & (x\beta)\alpha=x(\beta\alpha), (\alpha(\beta x)=(\alpha\beta)x) \\
(8) \ & x1=x,(1x=x).
\end{aligned}
\]
\end{definition}

\begin{definition}
	Let $\mathcal{V}$ be a right linear quaternionic vector space. A quaternionic inner product on $\mathcal{V}$ is a map $\langle\cdot,\cdot\rangle_{\mathbb{H}}:\mathcal{V} \times \mathcal{V}\mapsto\mathbb{H}$ satisfying the following conditions,
	\begin{flalign}
& (1)\langle\boldsymbol{x},\boldsymbol{y}\rangle_\mathbb{H}=\overline{\langle\boldsymbol{y},\boldsymbol{x}\rangle_\mathbb{H}}, \nonumber & \\
& (2)\langle\boldsymbol{x}p+\boldsymbol{y}q,\boldsymbol{z}\rangle_\mathbb{H}=\langle\boldsymbol{x},\boldsymbol{z}\rangle_\mathbb{H}p+\langle\boldsymbol{y},\boldsymbol{z}\rangle_\mathbb{H}q, \nonumber & \\
& (3)\langle\boldsymbol{x},\boldsymbol{y}p+\boldsymbol{z}q\rangle_\mathbb{H}=\overline{p}\langle\boldsymbol{x},\boldsymbol{y}\rangle_\mathbb{H}+\overline{q}\langle\boldsymbol{x},\boldsymbol{z}\rangle_\mathbb{H}, \nonumber & \\
& (4)\langle\boldsymbol{x},\ \boldsymbol{x}\rangle_\mathbb{H}\geq0\mathrm{\ and\ }\langle\boldsymbol{x},\boldsymbol{x}\rangle_{\mathbb{H}}=0\ \text{if and only if}\ \boldsymbol{x}=0, \nonumber
\end{flalign}
for all $\boldsymbol{x},\ \boldsymbol{y},\ \boldsymbol{z}\in\mathcal{V}$ and $p,\ q\in\mathbb{H}$. Obviously,
$||\boldsymbol{x}||_{\mathbb{H}}=\sqrt{\langle\boldsymbol{x},\boldsymbol{x}\rangle_\mathbb{H}}$
is a norm on $\mathcal{V}$. Furthermore, if the right linear quaternionic vector space $\mathcal{V}$ together with a quaternionic inner product is a complete normed vector space,  $\mathcal{V}$ is called a quaternionic Hilbert space. The quaternionic Hilbert space is separable if it has countable orthonormal basis.
	\end{definition}

Throughout the rest of this article, we denote by $\mathcal{H}_q$ a separable quaternionic Hilbert space.

\begin{definition}
A map $T:\ \mathcal{H}_q \mapsto \mathcal{H}_q$ is said to be a right linear quaternionic operator if for any $\boldsymbol{u},\ \boldsymbol{v}\in \mathcal{H}_q $ and any $s\in\mathbb{H}$,
\[
T(\boldsymbol{u}+\boldsymbol{v})=T(\boldsymbol{u})+T(\boldsymbol{v}),\ T(\boldsymbol{u}s)=T(\boldsymbol{u})s.
\]
Similarly, a map $T:\ \mathcal{H}_q  \mapsto \mathcal{H}_q $ is said to be a left linear quaternionic operator if
\[
T(\boldsymbol{u}+\boldsymbol{v})=T(\boldsymbol{u})+T(\boldsymbol{v}),\ T(s\boldsymbol{u})=sT(\boldsymbol{u}),
\]
for any $\boldsymbol{u},\ \boldsymbol{v}\in \mathcal{H}_q $ and any $s\in\mathbb{H}$.
Moreover, we say that $T$ is bounded if
$$\mathrm{sup}_{\substack{||\boldsymbol{x}||= 1\\ \boldsymbol{x}\in \mathcal{H}_q}} ||T\boldsymbol{x}||\le \infty.$$
Denote by $L_r(\mathcal{H}_q)$ ($L_l(\mathcal{H}_q)$) the set of all the bounded right (left) linear quaternionic operators on a separable quaternionic Hilbert space
$\mathcal{H}_q$.
\end{definition}

In this paper, we only consider the right linear quaternionic operators.

Notice that each element $a$ in $\mathbb{H}$ can be written as the following form,
\[
a=(a_0+a_1\boldsymbol{i})+\boldsymbol{j}(a_2+a_3\boldsymbol{i}).
\]
Given any $\boldsymbol{x}=(x_1,\ x_2,\ x_3\ \cdots)^t\in\mathbb{H}_q$, where $\boldsymbol{x}^t$ means the transposition of $\boldsymbol{x}$. We could write
\[
\boldsymbol{x}=\boldsymbol{x}_1+\boldsymbol{j}\boldsymbol{x}_2,
\]
where $\boldsymbol{x}_1=(x_{11},\ x_{21},\ x_{31},\ \cdots)^t$ and $\boldsymbol{x}_2=(x_{12},\ x_{22},\ x_{32}\ \cdots)^t$ are complex vectors. We call
\[
\boldsymbol{x}_\mathbb{C}=\begin{pmatrix}
	\boldsymbol{x}_1 \\
	\boldsymbol{x}_2
\end{pmatrix}
\]
the complex representation of the vector $\boldsymbol{x}$. Furthermore, we write
\[
(\mathcal{H}_q)_{\mathbb{C}}=\{\boldsymbol{x}_\mathbb{C}; ~\boldsymbol{x}\in\mathcal{H}_q\}.
\]

For every infinite or finite dimensional quaternionic matrix $A$, we can write
$$
A=
\begin{pmatrix}
    a_{11} & a_{12} & \cdots & a_{1n} & \cdots \\
    a_{21} & a_{22} & \cdots & a_{2n} & \cdots \\
    a_{31} & a_{32} & \cdots & a_{3n} & \cdots \\
    \vdots & \vdots & \ddots & \vdots & \cdots \\
\end{pmatrix}
=\begin{pmatrix}
    b_{11} & b_{12} & \cdots & b_{1n} & \cdots \\
    b_{21} & b_{22} & \cdots & b_{2n} & \cdots \\
    b_{31} & b_{32} & \cdots & b_{3n} & \cdots \\
    \vdots & \vdots & \ddots & \vdots & \cdots \\
\end{pmatrix}
+\boldsymbol{j}
\begin{pmatrix}
    c_{11} & c_{12} & \cdots & c_{1n} & \cdots \\
    c_{21} & c_{22} & \cdots & c_{2n} & \cdots \\
    c_{31} & c_{32} & \cdots & c_{3n} & \cdots \\
    \vdots & \vdots & \ddots & \vdots & \cdots \\
\end{pmatrix}
$$
where $a_{ij}=b_{ij}+\boldsymbol{j}c_{ij}$, $b_{ij},\ c_{ij}\in\mathbb{C}$ for all $i,\ j\in\mathbb{N}$. Then, $A=A_1+\boldsymbol{j}A_2$, where
\[A_1=\begin{pmatrix}
    b_{11} & b_{12} & \cdots & b_{1n} & \cdots \\
    b_{21} & b_{22} & \cdots & b_{2n} & \cdots \\
    b_{31} & b_{32} & \cdots & b_{3n} & \cdots \\
    \vdots & \vdots & \ddots & \vdots & \cdots \\
\end{pmatrix}
\ \ \ \text{and} \ \ \ A_2=\begin{pmatrix}
    c_{11} & c_{12} & \cdots & c_{1n} & \cdots \\
    c_{21} & c_{22} & \cdots & c_{2n} & \cdots \\
    c_{31} & c_{32} & \cdots & c_{3n} & \cdots \\
    \vdots & \vdots & \ddots & \vdots & \cdots \\
\end{pmatrix}
\]
are complex matrices. Let
\[
A_{\mathbb{C}}=
\begin{pmatrix}
A_1 & -\overline{A_2} \\
A_2 & \overline{A_1}\\
\end{pmatrix}.
\]
We call $A_\mathbb{C}$ the complex representation of $A$. Notice that the complex representations preserve the algebraic structure of quaternionic matrices. More precisely, we have $(A+B)_\mathbb{C}=A_\mathbb{C}+B_\mathbb{C}$, $(AB)_\mathbb{C}=A_\mathbb{C}B_\mathbb{C}$ and $( A\lambda)_{\mathbb{C}}=A_\mathbb{C}(\lambda\mathbf{I})_\mathbb{C}$,
for any quaternionic matrices $A$ and $B$ and any $\lambda\in\mathbb{H}$. On the contrary, for a complex matrix $M$ with the form
\[
M=\begin{pmatrix}
M_1 & -\overline{M_2} \\
M_2 & \overline{M_1} \\
\end{pmatrix},
\]
we call $M_\mathbb{H}=M_1+\boldsymbol{j}M_2$ is the quaternionic representation of $M$ and
$(M_{\mathbb{H}})_{\mathbb{C}}=M$. Obviously, for any quaternionic matrix $A$, we have $(A_{\mathbb{C}})_{\mathbb{H}}=A$.

Let $T$ be a bounded right linear quaternionic operator in $L_r(\mathcal{H}_q)$ and let $\{e_n\}_{n=1}^\infty$ be an orthonormal basis  of $\mathcal{H}_q$. Then, the matrix representation
of $T$ under $\{e_n\}_{n=1}^\infty$ is a quaternionic matrix, and we obtain the complex matrix representation $T_{\mathbb{C}}$ of $T$. Then, we also use $T_{\mathbb{C}}$ to denote the complex linear operator acting on $(\mathcal{H}_q)_{\mathbb{C}}$. Moreover, for any  $T\in L_r(\mathcal{H}_q)$ and any
$\boldsymbol{x}\in\mathbb{H}_q$, we have $(T\boldsymbol{x})_{\mathbb{C}}=T_{\mathbb{C}}\boldsymbol{x}_{\mathbb{C}}$.

\begin{definition}
Let $T_1$ and $T_2$ be two bounded right linear quaternionic operators in $L_r(\mathcal{H}_q)$. We say that $T_1$ is quaternion similar to $T_2$ if there exists an invertible right linear quaternionic operator $Q$ such that
$QT_1=T_2Q$. In particular, if $Q^*Q=QQ^*=\boldsymbol{I}$, where $\boldsymbol{I}$ means the identity operator on $\mathcal{H}_q$, we say that $T_1$ is quaternion unitarily equivalent to $T_2$.
\end{definition}

\begin{definition}\label{LRESP}
Let $T$ be a right linear quaternionic operator in $L_r(\mathcal{H}_q)$. The set of all right eigenvalues of $T$ is defined by
\[
\sigma_r(T)=\{\omega\in\mathbb{H};\ \text{there exists a nonzero vector}\ \boldsymbol{x}\ \text{such that }\ T\boldsymbol{x}=\boldsymbol{x}\omega\}.
\]
Here, $\boldsymbol{x}$ is called the right eigenvector of $T$ with respect to $\omega$.
Similarly, the set of all left eigenvalues of $T$ is defined by
\[
\sigma_l(T)=\{\lambda\in\mathbb{H};\ \text{there exists a nonzero vector}\ \boldsymbol{y}\ \text{such that }\ T\boldsymbol{y}=\lambda\boldsymbol{y}\}.
\]
Here, $\boldsymbol{y}$ is called the left eigenvector of $T$ with respect to $\lambda$.
\end{definition}
For convenience, we use $T-\mathbf{I}\omega$ to denote the map defined by,
\[
(T-\mathbf{I}\omega)\boldsymbol{x}=T\boldsymbol{x}-\boldsymbol{x}\omega, \ \ \  \  \text{for any} \ \boldsymbol{x}\in\mathcal{H}_q.
\]
Similarly, we use $T-\lambda\mathbf{I}$ to denote the map defined by ,
\[
(T-\lambda\mathbf{I})\boldsymbol{x}=T\boldsymbol{x}-\lambda\boldsymbol{x}, \ \ \  \  \text{for any} \ \boldsymbol{x}\in\mathcal{H}_q.
\]
Notice that $T-\lambda\mathbf{I}$ is a right linear quaternionic operator, but not is $T-\mathbf{I}\omega$ if $\omega\notin\mathbb{R}$.

In the study of noncommutative functional calculus for quaternionic operators, the $S$-spectrum plays an important role (see \cite{CSS} for example).
Now let us introduce some conclusions of  $S$-spectrum.

\begin{theorem}[$S$-spectrum radius formula \cite{CGSS}]\label{SMT}
	Let $T$ be a bounded right linear quaternionic operator in $L_r(\mathcal{H}_q)$. Then
	\[
	r_s(T)=\lim_{m\mapsto\infty}\Vert T^m \Vert^{{1}/{m}}.
	\]
\end{theorem}

Furthermore, we could define $S$-point spectrum as follows.

\begin{definition}
Let $T\in L_r(\mathcal{H}_q)$. The set of all $S$-point spectrum is defined by
	$$\sigma_{sp}(T)=\{s\in\mathbb{H};\ \exists\ 0\neq\boldsymbol{v}\in\mathcal{H}_q \ \text{such that}\ (T^2-2\mathrm{Re}(s)T+|s|^2\boldsymbol{I})\boldsymbol{v}=0\},$$
where $\boldsymbol{v}$ is called the $S$-eigenvector of $T$ with respect to $s$.
The $S$-point spectrum radius of $T$ is defined by
\[
r_{sp}(T):=\sup\{|s|;\ s\in\sigma_{sp}(T)\}.
\]
\end{definition}

In this paper, we will use quaternion-valued holomorphic functions on a connected open set in the (upper-half) complex plane. More about generalized quaternion-valued holomorphic functions, we also refer to \cite{CSS}.
\begin{definition}
Let $\Omega$ be a connected open set in $\mathbb{C}$. We say that a quaternion-valued function $f:\ \Omega\mapsto\mathbb{H}$ is right holomorphic if the following limit exists,
\[
\lim_{\Delta z\rightarrow0}(f(z+\Delta z)-f(z))(\Delta z)^{-1}.
\]
\end{definition}

\begin{remark}
Let $\Omega$ be a connected open set in $\mathbb{C}$ and let $f$ be a real differentiable quaternion-valued function from $\Omega$ to $\mathbb{H}$. Write $z=x+\boldsymbol{i}y\in\mathbb{C}$. Define
\[
\frac{\partial f}{\partial z}=\frac{1}{2}\left(\frac{\partial f}{\partial x}-\frac{\partial f}{\partial y}\cdot\boldsymbol{i} \right) \ \ \text{and} \ \ \
\frac{\partial f}{\partial \bar{z}}=\frac{1}{2}\left(\frac{\partial f}{\partial x}+\frac{\partial f}{\partial y}\cdot\boldsymbol{i} \right).
\]
Then the following  statements are equivalent.
\begin{enumerate}
	\item[(1)] \ $f$ is right holomorphic.
	\item[(2)] \ For any $z_0\in \Omega$, $f(z)$ has a Taylor expansion on a neighbourhood of $z_0$, i.e.,
 $f(z)=\sum_{n=0}^{\infty}\alpha_n(z-z_0)^n$,  where  $\alpha_n\in\mathbb{H}$.
	\item[(3)] \  $f(z)=f_1(z)+\boldsymbol{j}f_2(z)$, where $f_1(z)$ and $f_2(z)$ are complex-valued holomorphic functions.
	\item[(4)] \  $\frac{\partial f}{\partial \bar{z}}\equiv 0$.
\end{enumerate}
Moreover, we could also consider right holomorphic quaternionic vector-valued functions and their equivalent definitions in a similar way.
\end{remark}

\section{Cowen-Douglas operators on a quaternionic Hilbert space}\label{sec3}

In this section, we will study some fundamental properties of  Cowen-Douglas operators on a quaternionic Hilbert space. Some examples are also shown.

\subsection{$S$-point spectra and right eigenvalues}

\begin{lemma}\label{SpR}
Let $T \in L_r(\mathcal{H}_q)$. Then
\begin{enumerate}
	\item[(1)]  Both $\sigma_s(T)$ and $\sigma_{sp}(T)$ are axially symmetric sets. For any $\omega\in\sigma_{sp}(T)$, $\ker(T^2-2\mathrm{Re}(\omega)T+\vert\omega \vert^2)$ is a right linear quaternionic vector space and if $\omega\sim\omega'$, then $\ker(T^2-2\mathrm{Re}(\omega)T+\vert\omega \vert^2)=\ker(T^2-2\mathrm{Re}(\omega')T+\vert\omega' \vert^2)$.
	\item[(2)]  $\sigma_r(T)$ is an axially symmetric set. For any $\lambda\in\sigma_r(T)$, $\ker(T-\mathbf{I}\lambda)$ is a right linear vector space over the field $\mathbb{C}_{I_\lambda}$. Moreover, for any $0\neq q\in\mathbb{H}$, $\ker(T-\boldsymbol{I}q^{-1}\omega q)=\{\boldsymbol{x}q; \boldsymbol{x}\in\ker(T-\boldsymbol{I}\omega)\}$.
	\end{enumerate}
\end{lemma}
    \begin{proof}
  (1) Notice that $\omega\sim\omega'$ if and only if $\vert\omega\vert=\vert\omega'\vert$ and $\rm{Re}(\omega)=\rm{Re}(\omega')$. It follows from Definition \ref{LRESP} that $\sigma_s(T)$ is an axially symmetric set. Now, suppose that $\omega\in\sigma_{sp}(T)$.
  Let $\boldsymbol{x}$ be an $S$-eigenvector of $T$ respect to $S$-point spectrum $\omega$, i.e.,
  \[
  (T^2-2\mathrm{Re}(\omega)T+|\omega|^2)\boldsymbol{x}=0.
  \]
  For every $p\in\mathbb{H}$, we have
  \[
  (T^2-2\mathrm{Re}(s)T+|s|^2)\boldsymbol{y}p=0.
  \]
  Then, $\ker(T^2-2\mathrm{Re}(\omega)T+\vert\omega \vert^2)$ is a right linear quaternionic vector space.

  Moreover, if $\omega\sim\omega'$, then it follows from $\vert\omega\vert=\vert\omega'\vert$ and $\rm{Re}(\omega)=\rm{Re}(\omega')$ that
  \[
  \ker(T^2-2\mathrm{Re}(\omega)T+\vert\omega \vert^2)=\ker(T^2-2\mathrm{Re}(\omega')T+\vert\omega' \vert^2).
  \]

  (2) Give any $\omega\in\sigma_r(T)$. Let $\boldsymbol{x}$ be the right eigenvector of $T$ respect to right eigenvalue $\omega$, i.e., $T\boldsymbol{x}=\boldsymbol{x}\omega$.
  Then for any $p\in\mathbb{C}_{I\omega}$, we have
  \[
  T\boldsymbol{x}p=\boldsymbol{x}\omega p=\boldsymbol{x}p\omega.
  \]
  So $\ker(T-\mathbf{I}\omega)$ is a right linear vector space over the field $\mathbb{C}_{I_\lambda}$.

  Moreover, for any $0\neq q\in\mathbb{H}$, we have
  \[
  T(\boldsymbol{x}q)=T(\boldsymbol{x})q=\boldsymbol{x}\omega q=\boldsymbol{x}q(q^{-1}\omega q).
  \]
  Therefore,
  \[
  \ker(T-\boldsymbol{I}q^{-1}\omega q)=\{\boldsymbol{x}q; \boldsymbol{x}\in\ker(T-\boldsymbol{I}\omega)\}.
  \]
    \end{proof}

\begin{lemma}
	Let $T_1$ and $T_2$ be two bounded right linear operators in $L_r(\mathcal{H}_q)$. If $T_1$ is quaternionic similar to $T_2$, then
    $\sigma_{sp}(T_1)=\sigma_{sp}(T_2)$ and $\sigma_r(T_1)=\sigma_r(T_2)$.
	\end{lemma}
	
	\begin{proof}
	Let $Q$ be an invertible operator in $L_r(\mathcal{H}_q)$ such that $Q^{-1}T_1Q=T_2$.
	
	For any $s\in\sigma_{sp}(T_1)$, let	 $\boldsymbol{v}$ be a nonzero vector
	such that $(T_1^2-2\mathrm{Re}(s)T_1+|s|^2)\boldsymbol{v}=0$. Then,
	\[
	\begin{aligned}
0&=Q^{-1}(T_1^2-2\mathrm{Re}(s)T_1+|s|^2)\boldsymbol{v}\\
&=Q^{-1}(T_1^2-2\mathrm{Re}(s)T_1+|s|^2)QQ^{-1}\boldsymbol{v}\\
&=(Q^{-1}T_1QQ^{-1}T_1Q-2\mathrm{Re}(s)Q^{-1}T_1Q+|s|^2)Q^{-1}\boldsymbol{v}\\
&=((Q^{-1}T_1Q)^2-2\mathrm{Re}(s)Q^{-1}T_1Q+|s|^2)Q^{-1}\boldsymbol{v} \\
&=(T_2^2-2\mathrm{Re}(s)T_2+|s|^2)(Q^{-1}\boldsymbol{v}).
\end{aligned}
\]
This implies $\sigma_{sp}(T_1)\subseteq\sigma_{sp}(T_2)$. In the same way, one can see that $\sigma_{sp}(T_2)\subseteq\sigma_{sp}(T_1)$.
Therefore, we have $\sigma_{sp}(T_1)=\sigma_{sp}(T_2)$.

For any $\omega\in\sigma_r(T_1)$, let $\boldsymbol{v}$ be a nonzero vector such that $T_1\boldsymbol{x}=\boldsymbol{x}\omega$. Then,
\[
(Q^{-1}\boldsymbol{x})\omega=Q^{-1}\boldsymbol{x}\omega=Q^{-1}(T_1\boldsymbol{x})=Q^{-1}T_1QQ^{-1}\boldsymbol{x}
=T_2(Q^{-1}\boldsymbol{x}).
\]
This implies  $\sigma_{r}(T_1)\subseteq\sigma_{r}(T_2)$. In the same way, one can see that $\sigma_{r}(T_2)\subseteq\sigma_{r}(T_1)$.
Therefore, we have $\sigma_{r}(T_1)=\sigma_{r}(T_2)$.
	\end{proof}

	\begin{theorem}\label{RSPE}
	Let $T\in L_r(\mathcal{H}_q)$. For any $s\in\mathbb{H}$,
	\[
	T^2-2\mathrm{Re}(s)T+|s|^2=(T-\mathbf{I}s)(T-\mathbf{I}\overline{s}).
	\]
	Furthermore, $\sigma_r(T)=\sigma_{sp}(T)$. Moreover, for any $\omega\in\sigma_{sp}(T)\setminus\mathbb{R}=\sigma_r(T)\setminus\mathbb{R}$, we have
\[
\ker(T^2-2\mathrm{Re}(\omega)T+|\omega|^2)=\bigvee_{\mathbb{H}} \ker(T-\boldsymbol{I}\omega).
\]
	\end{theorem}
	
	\begin{proof}
	For any $\boldsymbol{y}\in\mathcal{H}_q$,
	\[
	\begin{aligned}
		(T^2-2\mathrm{Re}(s)T+|s|^2)\boldsymbol{y}=&T^2\boldsymbol{y}-2\mathrm{Re}(s)T\boldsymbol{y}+|s|^2\boldsymbol{y}\\
		=&T^2\boldsymbol{y}-T\boldsymbol{y}\cdot2\mathrm{Re}(s)+|s|^2\boldsymbol{y}\\
		=&T^2\boldsymbol{y}-T\boldsymbol{y}(s+\overline{s})+\boldsymbol{y}s\overline{s} \\
		=&(T-\mathbf{I}s)(T-\mathbf{I}\overline{s})\boldsymbol{y}.
	\end{aligned}
	\]
Furthermore, it is obvious that $\sigma_r(T)\subseteq\sigma_{sp}(T)$.
On the other hand, let $s\in\sigma_{sp}(T)$ and let $\boldsymbol{x}$ be a nonzero vector such that $(T^2-2\mathrm{Re}(s)T+|s|^2)\boldsymbol{x}=0$. Then,
$(T-\mathbf{I}s)(T-\mathbf{I}\overline{s})\boldsymbol{x}=0$
and consequently at least one of $s$ and $\overline{s}$ is a right eigenvalue of $T$. Since $s\sim\overline{s}$, we always have $s\in\sigma_{r}(T)$. Therefore, $\sigma_{sp}(T)=\sigma_r(T)$.

Given any $\omega\in\sigma_{sp}(T)\setminus\mathbb{R}=\sigma_r(T)\setminus\mathbb{R}$. It is obvious that $\ker(T-\boldsymbol{I}\omega)\subseteq\ker(T^2-2\mathrm{Re}(\omega)T+|\omega|^2)$. Since $\ker(T^2-2\mathrm{Re}(\omega)T+|\omega|^2)$  is a right linear quaternionic vector space, we have
\[
\bigvee_{\mathbb{H}} \ker(T-\boldsymbol{I}\omega)\subseteq\ker(T^2-2\mathrm{Re}(\omega)T+|\omega|^2).
\]
Now, let $\boldsymbol{x}$ be a nonzero vector
such that $(T^2-2\mathrm{Re}(\omega)T+|\omega|^2)\boldsymbol{x}=0$. Since
\[
(T-\mathbf{I}\omega)(T-\mathbf{I}\overline{\omega})\boldsymbol{x}=(T-\mathbf{I}\overline{\omega})(T-\mathbf{I}\omega)\boldsymbol{x}=0,
\]
we have
\[
(T-\mathbf{I}\overline{\omega})\boldsymbol{x}\in\ker(T-\mathbf{I}\omega) \ \ \text{and} \ \ (T-\mathbf{I}\omega)\boldsymbol{x}\in\ker(T-\mathbf{I}\overline{\omega}).
\]
Notice that there exists $0\neq q\in\mathbb{H}$ such that $\omega=q^{-1}\overline{\omega}q$.
By Lemma \ref{SpR}, there exist two vectors $\boldsymbol{u}, \boldsymbol{v}\in \ker(T-\mathbf{I}\omega)$ such that
\[
\boldsymbol{u}=(T-\mathbf{I}\overline{\omega})\boldsymbol{x} \ \ \text{and} \ \ \boldsymbol{v}q=(T-\mathbf{I}\omega)\boldsymbol{x}.
\]
Consequently,
\[
\boldsymbol{u}-\boldsymbol{v}q=\boldsymbol{x}(\omega-\overline{\omega}),
\]
and then
\[
\boldsymbol{x}=(\boldsymbol{u}-\boldsymbol{v}q)(\omega-\overline{\omega})^{-1}\in\bigvee_{\mathbb{H}} \ker(T-\boldsymbol{I}\omega).
\]
Therefore,
\[
\ker(T^2-2\mathrm{Re}(\omega)T+|\omega|^2)\subseteq\bigvee_{\mathbb{H}} \ker(T-\boldsymbol{I}\omega).
\]
This finishes the proof.
	\end{proof}

\subsection{Fundamental properties of Cowen-Douglas operators on a separable quaternionic Hilbert space}

Let $\Omega_q$ be a connected open subset in $\mathbb{H}$ and  $n$ be a positive integer. Recall that  $T\in B_n^{s}(\Omega_q)$ if the following conditions hold.
\begin{enumerate}
\item[(a)] \ $\Omega_q\subseteq\sigma_{s}(T)$; \\
\item[(b)] \ $\mathrm{ran}(T^2-2\mathrm{Re}(\omega) T+\vert\omega \vert^2)=\mathcal{H}_q$, for any $\omega\in\Omega_q$;\\
\item[(c)] \ $\bigvee_{\mathbb{H}} \{\ker(T^2-2\mathrm{Re}(\omega) T+\vert\omega \vert^2);\ \omega\in\Omega_q\}=\mathcal{H}_q$;\\
\item[(d)] \ $\dim_{\mathbb{H}} \ker(T^2-2\mathrm{Re}(\omega)T+\vert\omega \vert^2)=n$, for any $\omega\in\Omega_q$.
\end{enumerate}

Review that  the axially symmetric reduced subset of $\Omega_q$ is
\[\Omega_{red}=\{\omega_{red}=\mathrm{Re}(\omega)+\boldsymbol{i}|\mathrm{Im}(\omega)|;\ \omega\in\Omega_q\}.
\]
For convenience, {\bf we assume that $\Omega_q\bigcap\mathbb{R}=\emptyset$}. Then, by Lemma $\ref{SpR}$, the conditions (a)--(d) are equivalent to the following conditions (a')--(d').
\begin{enumerate}
	\item[(a')] \ $\Omega_{red}\subseteq\sigma_{s}(T)$; \\
	\item[(b')] \ $\mathrm{ran}(T^2-2\mathrm{Re}(\omega) T+\vert\omega \vert^2)=\mathcal{H}_q$, for any $\omega\in\Omega_{red}$;\\
	\item[(c')] \ $\bigvee_{\mathbb{H}} \{\ker(T^2-2\mathrm{Re}(\omega) T+\vert\omega \vert^2);\ \omega\in\Omega_{red}\}=\mathcal{H}_q$;\\
	\item[(d')] \ $\dim_{\mathbb{H}} \ker(T^2-2\mathrm{Re}(\omega)T+\vert\omega \vert^2)=n$, for any $\omega\in\Omega_{red}$.
\end{enumerate}
We often use the symbol $B_n^{s}(\Omega_{red})$ instead of $B_n^{s}(\Omega_q)$ without confusion.
We will show that each operator $T\in B_n^{s}(\Omega_q)$ induces an $n$-dimensional quaternionic vector bundle over $\Omega_{red}$. Denote by $\mathcal{G}r(n,\ \mathcal{H}_q)$  the $n$-Grassmann manifold, which is the set of all $n$-dimensional subspaces of $\mathcal{H}_q$.

\begin{definition}
	A map $t:\Omega_{red}\rightarrow\mathcal{G}r(n,\ \mathcal{H}_q)$ is called a right holomorphic curve, if for every $\omega_0\in\Omega_{red}$, there exists an open neighborhood $\Delta$ of $\omega_0$ and   right holomorphic mappings $\gamma_i:\Delta\rightarrow\mathcal{H}_q$, $i=1, \cdots, n$, such that
    $\{\gamma_i(\omega)\}_{i=1}^n$ is a basis of the $n$-dimensional subspace $t(\omega)$ for any $\omega\in\Delta$. We call $\{\gamma_i\}_{i=1}^n$ a right holomorphic frame of $t$ over $\Delta$. Moreover, a right holomorphic mapping $\gamma:\Delta\rightarrow\mathcal{H}_q$ is said to be a right holomorphic cross-section of $t$ if $\gamma(\omega)\in t(\omega)$ for any $\omega\in\Delta$.

    Furthermore, a right holomorphic curve $t:\Omega_{red}\rightarrow\mathcal{G}r(n,\ \mathcal{H}_q)$ induces an $n$-dimensional Hermitian right holomorphic quaternionic vector bundle $(E_t,\ \Omega_{red},\pi)$ which is  defined by
    \[
    E_t=\{(\boldsymbol{x};\ \omega),\ \boldsymbol{x}\in t(\omega),\ \omega\in\Omega_{red}\}\ \mathrm{and}\ \pi(\boldsymbol{x},\ \omega)=\omega.
    \]
    \end{definition}

    Let  $T\in B_n^{s}(\Omega_{red})$. Then $T$ induces a curve $t_T:\Omega_{red}\rightarrow\mathcal{G}r(n,\ \mathcal{H}_q)$ defined by
\[
t_T(\omega)=\ker(T^2-2\mathrm{Re}(\omega)T+\vert\omega\vert^2),\ \ \ \text{for any} \ \omega\in\Omega_{red},
\]
and then $t_T$ induces a vector bundle $(E_T,\ \Omega_{red},\pi)$  which is defined by
  \[
E_T=\{(\boldsymbol{x};\ \omega),\ \boldsymbol{x}\in \ker(T^2-2\mathrm{Re}(\omega)T+\vert\omega\vert^2),\ \omega\in\Omega_{red}\}\ \mathrm{and}\ \pi(\boldsymbol{x},\ \omega)=\omega.
\]

Now, we will show that the curve $t_T$ induced by $T\in B_n^{s}(\Omega_{red})$ is a right holomorphic curve.

\begin{lemma}[\cite{CO}]\label{CO}
	Let $\mathcal{X}$ be a Banach space and $A\in\mathcal{X}$. Then $A$ is right invertible if and only if $ran A = \mathcal{X}$ and $\ker A$ is a complemented subspace of $\mathcal{X}$.
	\end{lemma}
	
	\begin{lemma}\label{RI}
	Let $T\in B_n^{s}(\Omega_{red})$. Then $t_T:\Omega_{red}\rightarrow\mathcal{G}r(n,\ \mathcal{H}_q)$ is a right holomorphic curve.
	\end{lemma}\label{RH}

\begin{proof}
Given any $\omega_0\in \Omega_{red}$. Let $\{x_1,\ x_2,\cdots,\ x_n\}$ be a complex basis of $\ker(T-\boldsymbol{I}\omega_0)$. Following from Theorem \ref{RSPE}, $\{x_1,\ x_2,\cdots,\ x_n\}$ is also a quaternionic basis of $\ker(T^2-2\mathrm{Re}(\omega_0)T+\vert\omega_0\vert^2)$.
Notice that $(T-\boldsymbol{I}\omega_0)_{\mathbb{C}}=T_{\mathbb{C}}-\omega_0\boldsymbol{I}_{\mathbb{C}}$ is a linear operator on $(\mathcal{H}_q)_{\mathbb{C}}$, although $T-\boldsymbol{I}\omega_0$ is not a quaternionic (right or left) linear operator on $\mathcal{H}_{q}$.
Then, by Lemma \ref{CO}, there exists a bounded linear operator $X$ on $(\mathcal{H}_q)_{\mathbb{C}}$ such that $(T-\boldsymbol{I}\omega_0)_{\mathbb{C}}X=I$.
For any $i=1,\cdots, n$, we define
\[
\gamma_i(\omega)=\sum_{k=0}^{\infty}(X^k((x_i)_{\mathbb{C}}))_{\mathbb{H}}(\omega-\omega_0)^{k},
\]
which is a right holomorphic cross-section over a neighborhood of $\omega_0$.

Since
\[
\begin{aligned}
    & (T^2-2\mathrm{Re}(\omega)T+|\omega|^2)_{\mathbb{C}}(\sum_{k=0}^{\infty}X^k((x_i)_{\mathbb{C}})(\omega-\omega_0)^k) \\
    = & (T_{\mathbb{C}}-\overline{\omega}\boldsymbol{I}_{\mathbb{C}})(T_{\mathbb{C}}-\omega\boldsymbol{I}_{\mathbb{C}})(\sum_{k=0}^{\infty}X^k((x_i)_{\mathbb{C}})(\omega-\omega_0)^k) \\
    = & (T_{\mathbb{C}}-\overline{\omega}\boldsymbol{I}_{\mathbb{C}})(T_{\mathbb{C}}-\omega_0\boldsymbol{I}_{\mathbb{C}})(\sum_{k=0}^{\infty}X^k((x_i)_{\mathbb{C}})(\omega-\omega_0)^k) \\
    &+(T_{\mathbb{C}}-\overline{\omega}\boldsymbol{I}_{\mathbb{C}})(\omega_0\boldsymbol{I}_{\mathbb{C}}-\omega\boldsymbol{I}_{\mathbb{C}})(\sum_{k=0}^{\infty}X^k((x_i)_{\mathbb{C}})(\omega-\omega_0)^k) \\
    = & (T_{\mathbb{C}}-\overline{\omega}\boldsymbol{I}_{\mathbb{C}})\left(\sum_{k=0}^{\infty}X^k((x_i)_{\mathbb{C}})(\omega-\omega_0)^{k+1}-\sum_{k=0}^{\infty}X^k((x_i)_{\mathbb{C}})(\omega-\omega_0)^{k+1}\right) \\
    = & 0,
\end{aligned}
\]
we have
\[
(T^2-2\mathrm{Re}(\omega)T+|\omega|^2)\gamma_i(\omega)=0.
\]
More precisely, we have $(T-\boldsymbol{I}\omega)\gamma_i(\omega)=0$.

Let
\[
f(\omega)=\begin{pmatrix}
\langle\gamma_1(\omega),\ \gamma_1(\omega)\rangle & \langle\gamma_1(\omega),\ \gamma_2(\omega)\rangle & \cdots & \langle\gamma_1(\omega),\ \gamma_n(\omega)\rangle \\
\langle\gamma_2(\omega),\ \gamma_1(\omega)\rangle & \langle\gamma_2(\omega),\ \gamma_2(\omega)\rangle & \cdots & \langle\gamma_2(\omega),\ \gamma_n(\omega)\rangle \\
\vdots & \vdots & \ddots & \vdots \\
\langle\gamma_n(\omega),\ \gamma_1(\omega)\rangle & \langle\gamma_n(\omega),\ \gamma_2(\omega)\rangle & \cdots & \langle\gamma_n(\omega),\ \gamma_n(\omega)\rangle \\
\end{pmatrix}.
\]
Since $\{\gamma_i(\omega_0)=x_i\}_{i=1}^{n}$ is linear independent, one can see that $f(\omega_0)$ is a positive definite quadratic form. By the continuity of $f(\omega)$ with respect to $\omega$, there exists a neighborhood $\Delta$ of $\omega_0$ such that $f(\omega)$ is a positive definite quadratic form for every $\omega\in\Delta$. Consequently, $\{\gamma_i(\omega)\}_{i=1}^{n}$ is linear independent. Together with $\dim_{\mathbb{H}}{\ker(T-\boldsymbol{I}\omega)}=n$, $\{\gamma_i(\omega)\}_{i=1}^{n}$ is a basis of $\ker(T-\boldsymbol{I}\omega)$ and  hence
$\{\gamma_i\}_{i=1}^{n}$ is a right holomorphic frame of $t_T$ over $\Delta$. Therefore, $t_T$ is a right holomorphic curve.
		\end{proof}

\begin{lemma}\label{CDP}
Let $\Delta$ be a connected open subset of
$\Omega_{red}$. Then $B_n^{s}(\Omega_{red})\subseteq B_n^{s}(\Delta)$.
\end{lemma}
    \begin{proof}
	Following from Definition $\ref{SPCD}$, it suffices to prove that
\[
\bigvee_{\mathbb{H}}\{\ker(T^2-2\mathrm{Re}(\omega)+\vert\omega\vert^2),\ \omega\in\Delta\}=\mathcal{H}_q.
\]
Suppose that $\bigvee_{\mathbb{H}}\{\ker(T^2-2\mathrm{Re}(\omega)+\vert\omega\vert^2);\ \omega\in\Delta\}\ne\mathcal{H}_q$. Then, there exists
a non-trivial vector $\boldsymbol{x}\in\mathcal{H}_q$ such that
\[
\langle\boldsymbol{x},\ \boldsymbol{v}\rangle=0, \ \ \ \text{for every} \ \boldsymbol{v}\in \bigvee_{\mathbb{H}}\{\ker(T^2-2\mathrm{Re}(\omega)+\vert\omega\vert^2),\ \omega\in\Delta\}.
\]
Given any right holomorphic cross-section $\gamma(\omega)$ over $\Omega_{red}$. Since $\langle\gamma(\omega),\ \boldsymbol{x}\rangle_{\mathbb{H}}$ is a right holomorphic function on $\Omega_{red}$ and $\langle\gamma(\lambda),\ \boldsymbol{x}\rangle_{\mathbb{H}}=0$ for any $\lambda\in\Delta$, it follows from the uniqueness of analytic functions that
$\langle\gamma(\omega),\ \boldsymbol{x}\rangle_{\mathbb{H}}=0$
for every $\omega\in\Omega_{red}$. This implies
\[
\langle\boldsymbol{v},\ \boldsymbol{x},\ \rangle_{\mathbb{H}}=0, \ \ \ \text{for every} \ \boldsymbol{v}\in \bigvee_{\mathbb{H}}\{\ker(T^2-2\mathrm{Re}(\omega)+\vert\omega\vert^2);\ \omega\in\Omega_{red}\},
\]
which is a contradiction to $\bigvee_{\mathbb{H}}\{\ker(T^2-2\mathrm{Re}(\omega)+\vert\omega\vert^2);\ \omega\in\Omega_{red}\}=\mathcal{H}_q$. This finishes the proof.
    \end{proof}

In fact, following from a similar  manner in the above proof, we could obtain the following result immediately.
\begin{corollary}
Let $T\in B_n^{s}(\Omega_{red})$. If a sequence of points $\{\omega_n\}_{n=1}^{\infty}$ in $\Omega_{red}$ has an accumulation point in $\Omega_{red}$,
then
\[
\bigvee_{\mathbb{H}}\{\ker(T^2-2\mathrm{Re}(\omega_n)T+\vert\omega_n\vert^2);\ n=1, 2,  \cdots\}=\mathcal{H}_q.
\]
\end{corollary}

\begin{lemma}\label{orthonormal basis P}
Let $T\in B_n^{s}(\Omega_{red})$. Suppose that $\gamma(\omega)$ is a right holomorphic cross-section of $t_T$.
Then, $\gamma(\omega)\in\ker(T-\boldsymbol{I}\omega)$ for any $\omega\in\Omega_{red}$. Furthermore, we have for any $k\in\mathbb{N}$,
\[
T\gamma^{(k)}(\omega)=\gamma^{(k)}(\omega)\omega+k\gamma^{(k-1)}(\omega)
\]
and consequently,
\begin{equation*}\label{nd}
(T^2-2\mathrm{Re}(\omega)T+\vert\omega\vert^2)\gamma^{(k)}(\omega) \\
=k(k-1)\gamma^{(k-2)}(\omega)+k\gamma^{(k-1)}(\omega)(\omega-\overline{\omega}),
\end{equation*}
where $\gamma^{(k)}(\omega)$ means the $k$-th derivative of $\gamma(\omega)$.
\end{lemma}

\begin{proof}
	 Since
	 \[
	 (T^2-2\mathrm{Re}(\omega)T+\vert\omega\vert^2)\gamma(\omega)=0,
	 \]
	 one can see that
 \[
	\begin{aligned}
		0=	& \frac{\partial}{\partial \overline{w}}\left((T^2-2\mathrm{Re}(\omega)T+\vert\omega\vert^2)\gamma(\omega)\right) \\
		= & \frac{\partial}{\partial \overline{w}}\left(T^2(\gamma(\omega))-T(\gamma(\omega)\omega)-T(\gamma(\omega)\overline{\omega})+\gamma(\omega)\omega \overline{\omega}\right) \\
		= & -T\gamma(\omega)+\gamma(\omega)\omega,
	\end{aligned}
	\]
	which implies  $\gamma(\omega)\in\ker(T-\boldsymbol{I}\omega)$ for any $\omega\in\Omega_{red}$.
	Then, following from $T\gamma(\omega)=\gamma(\omega)\omega$, we have that
	\[
	T\gamma^{(k)}(\omega)=\gamma^{(k)}(\omega)\omega+k\gamma^{(k-1)}(\omega).
	\]
Consequently,
	 \[
	 \begin{aligned}
	 	& (T^2-2\mathrm{Re}(\omega)T+\vert\omega\vert^2)\gamma^{(k)}(\omega) \\
	 	= & (T-\boldsymbol{I}\overline{\omega})(T-\boldsymbol{I}\omega)\gamma^{(k)}(\omega) \\
	 	= & ((T-\boldsymbol{I}\omega)+(\boldsymbol{I}\omega-\boldsymbol{I}\overline{\omega}))(k\gamma^{(k-1)}(\omega)) \\
	 	=& k(k-1)\gamma^{(k-2)}(\omega)+k\gamma^{(k-1)}(\omega)(\omega-\overline{\omega}).
	 \end{aligned}
	 \]
The proof is finished.
    \end{proof}
\begin{remark}\label{orthonormal basis PC}
Following from Lemma \ref{orthonormal basis P}, one can see that for any $k\in\mathbb{N}$,
\[
\begin{aligned}
(T^2-2\mathrm{Re}(\omega)T+\vert\omega\vert^2)^k\gamma^{(j)}(\omega)=
	\begin{cases}
		0,\ & 0\leq j\leq k-1 \\
		 k!\gamma(\omega)(\omega-\overline{\omega})^k,\ & j=k
	\end{cases}
\end{aligned},
\]
which implies $\gamma^{(j)}(\omega)\in\bigvee_{k=1}^{\infty}\ker(T^2-2\mathrm{Re}(\omega_0)T+\vert\omega_0\vert^2)^k$ for $j=0, \cdots, k-1$. Given any $\omega_0\in\Omega_{red}$. Notice that each right holomorphic cross-section has a Taylor expansion over a neighborhood $\Delta$ of $\omega_0$. Then, together with Lemma \ref{CDP}, we have
\[
\bigvee_{k=1}^{\infty}\ker(T^2-2\mathrm{Re}(\omega_0)T+\vert\omega_0\vert^2)^k=\bigvee_{\mathbb{H}}\{\ker(T^2-2\mathrm{Re}(\omega)+\vert\omega\vert^2),\ \omega\in\Delta\}=\mathcal{H}_q.
\]
In fact, it is not difficult to see that the third condition of Definition $\ref{SPCD}$ is equivalent to
\[
\bigvee_{k=1}^{\infty}\ker(T^2-2\mathrm{Re}(\omega_0)T+\vert\omega_0\vert^2)^k=\mathcal{H}_q.
\]
\end{remark}

\begin{lemma}\label{Basis}
Let $T\in B_n^{s}(\Omega_{red})$ and let $\{\gamma_i\}_{i=1}^{n}$ be a right holomorphic frame of the vector bundle $E_T$. Given any $\omega_0\in\Omega_{red}(\subset \mathbb{C}\setminus\mathbb{R})$. Then for any $k\in\mathbb{N}$,
\[
\{\gamma_1(\omega_0),\ \cdots,\ \gamma_n(\omega_0),\ \cdots,\ \gamma_1^{(k-1)}(\omega_0),\ \cdots,\ \gamma_n^{(k-1)}(\omega_0)\}
\]
is a basis of $\ker(T^2-2\mathrm{Re}(\omega_0)T+\vert\omega_0\vert^2)^k$.
\end{lemma}
    \begin{proof}
    It suffices to prove that $[{{\gamma_1^{(k-1)}(\omega_0)}}],\ \cdots,\ [{\gamma_n^{(k-1)}(\omega_0)}]$
    is s basis of
    \[
    \ker(T^2-2\mathrm{Re}(\omega_0)T+\vert\omega_0\vert^2)^k\slash\ker(T^2-2\mathrm{Re}(\omega_0)T+\vert\omega_0\vert^2)^{k-1},
    \]
    where $[\ \cdot\ ]$ means the
    residue class. If $\sum_{i=1}^n[\gamma_i^{(k-1)}(\omega_0)]a_i=0$, $a_i\in\mathbb{H}$, then
    \[
    (T^2-2\mathrm{Re}(\omega_0)T+|\omega_0|^2)^{k-1}\sum_{i=1}^n\gamma_i^{(k-1)}(\omega_0)a_i=0.
    \]
    Consequently, by Lemma \ref{orthonormal basis P} (or Remark \ref{orthonormal basis PC}), we have
    \[
   \sum_{i=1}^n(k-1)!\gamma_i(\omega_0)(\omega-\overline{\omega})^ka_i=0.
    \]
    Notice that $\omega_0\in\Omega_{red}\subset \mathbb{C}\setminus\mathbb{R}$.
    Then, $a_i=0$ for $i=1,\ \cdots,\ n$ and hence the elements in the sequence $\{[{{\gamma_1^{(k-1)}(\omega_0)}}],\ \cdots,\ [{\gamma_n^{(k-1)}(\omega_0)}]\}$
    are quaternionic right linear independent.

    Since $T^2-2\mathrm{Re}(\omega_0)T+\vert\omega_0\vert^2$ is a Fredholm operator and $\mathrm{ran}(T^2-2\mathrm{Re}(\omega_0)T+\vert\omega_0\vert^2)=\mathcal{H}_q$,
    we have
    \[
    \mathrm{dim}\ker(T^2-2\mathrm{Re}(\omega_0)T+\vert\omega_0\vert^2)=nk,
    \]
    and consequently,
    \[
    \mathrm{dim}(\ker(T^2-2\mathrm{Re}(\omega_0)T+|\omega_0|^2)^k\slash\ker(T^2-2\mathrm{Re}(\omega_0)T+\vert\omega_0\vert^2)^{k-1})=n.
    \]
    Therefore, $[{{\gamma_1^{(k-1)}(\omega_0)}}],\ \cdots,\ [{\gamma_n^{(k-1)}(\omega_0)}]$
    is s basis of
    \[
    \ker(T^2-2\mathrm{Re}(\omega_0)T+\vert\omega_0\vert^2)^k\slash\ker(T^2-2\mathrm{Re}(\omega_0)T+\vert\omega_0\vert^2)^{k-1}.
    \]
    This completes the proof.
    \end{proof}

    \begin{remark}
	In particular, if $T\in B_1^{sp}(\Omega_{red})$ and $\gamma(\omega)$ is a right holomorphic  frame of $E_T$ over a neighborhood $\Delta$ of $\omega\in\Omega_{red}$. Then,
	\[
	\{\gamma(\omega),\ \gamma'(\omega),\ \cdots,\ \gamma^{(n-1)}(\omega)\}
	\]
	is a basis of $\ker(T^2-2\mathrm{Re}(\omega)T+\vert\omega\vert^2)^{n}$.
	Furthermore, the operator  $(T^2-2\mathrm{Re}(\omega)T+|\omega|^2)$ acting on the invariant subspace $\ker(T^2-2\mathrm{Re}(\omega)T+\vert\omega\vert^2)^{n}$, under the basis $\{\gamma(\omega),\ \gamma'(\omega),\ \cdots,\ \gamma^{(n-1)}(\omega)\}$,
	 has the following matrix representation
	\[
	\begin{pmatrix}
		0 &	\omega-\overline{\omega} & 2 & 0 & 0 & \cdots \\
		0 &  0                    &2(\omega-\overline{\omega}) & 6  & 0 & \cdots \\
		0 & 0  &0 & 3(\omega-\overline{\omega}) & 12 & \cdots \\
		0 & 0  &0 &  0 & 4(\omega-\overline{\omega})  & \cdots \\
		\vdots & \vdots & \vdots & \vdots & \vdots & \ddots
	\end{pmatrix}.
\]
\end{remark}

\subsection{Examples: backward unilateral weighted shift operators}

In this subsetion, we study the  backward unilateral weighted shift operators in the class of Cowen-Douglas operators on quaternionic Hilbert spaces. We always assume that $\{\mathbf{e}_n\}_{n=1}^{\infty}$ is an orthonormal basis  of a quaternionic Hilbert space $\mathcal{H}_q$. The backward unilateral weighted shift operator $T:\mathcal{H}_q\rightarrow\mathcal{H}_q$ is defined by
	\[
T\mathbf{e}_{1}=0\ \ \text{and} \ \ 	T\mathbf{e}_{n+1}=\mathbf{e}_{n}w_n,\ \ \ w_n\in\mathbb{H}\setminus\{0\},\ \text{for} \ n=1,\ 2,\ \cdots.
	\]
Furthermore,
$$T(\sum_{n=1}^{\infty}e_n\alpha_n) = \sum_{n=1}^{\infty}e_nw_n\alpha_{n+1}$$
where $\alpha_n\in\mathbb{H},\ n=1,\ 2,\ \cdots $.

The sequence $\{w_n\}_{n=1}^{\infty}$ is called the weight sequence of $T$. Obviously, $T$ is bounded if and only if its weight sequence is bounded. Moreover, since
\[
\|T^n\|=\sup\limits_{i\in\mathbb{N}} |w_{i+1}w_{i+2}\cdots w_{i+n}|,
\]
it follows from Theorem \ref{SMT} that
\[
r_s(T)=\lim\limits_{n\to\infty}\sup\limits_{i\in\mathbb{N}} |w_{i+1} w_{i+2}\cdots w_{i+n}|^{1/n}.
\]

	\begin{lemma}\label{BUWSS}
	Suppose that $T \in L_r(\mathcal{H}_q)$ is a backward unilateral weighted shift operator with the weight sequence $\{w_n\}_{n=1}^{\infty}$. 	Let $\widetilde{T}$ be the backward unilateral weighted shift operator with the weight sequence $\{|w_n|\}_{n=1}^{\infty}$. Then, $T$ is unitarily equivalent to $\widetilde{T}$.
	\end{lemma}

    \begin{proof}
    Let $X$ be the diagonal operator defined by
    \[
    X\mathbf{e}_1=\mathbf{e}_1 \ \ \text{and} \ \ X\mathbf{e}_{n}=\mathbf{e}_n\cdot \frac{\bar{w}_{n-1}\bar{w}_{n-2}\cdots\bar{w}_1}{|w_{n-1}||w_{n-2}|\cdots|w_1|},\ \ \text{for} \ n\geq 2.
    \]
    It is easy to see that $X$ is a unitary operator and $TX\mathbf{e}_1=0=X\widetilde{T}\mathbf{e}_1$,
    \[
    TX\mathbf{e}_{n+1}=\mathbf{e}_{n}\cdot \frac{{w}_{n}\bar{w}_{n}\bar{w}_{n-1}\cdots\bar{w}_1}{|w_{n}||w_{n-1}|\cdots|w_1|}=\mathbf{e}_{n}\cdot \frac{|{w}_{n}|\bar{w}_{n-1}\cdots\bar{w}_1}{|w_{n-1}|\cdots|w_1|}=X\widetilde{T}\mathbf{e}_{n+1}, \ \ \text{for} \ n\geq 1.
    \]
Then, we have $TX=X\widetilde{T}$ and hence $T$ is unitarily equivalent to $\widetilde{T}$.
    \end{proof}	

Following the above result, we would only consider backward unilateral weighted shift operators with positive weight sequences without loss of generality.

\begin{lemma}\label{BUSEV}
Let $T\in L_r(\mathcal{H}_q)$ be a backward unilateral weighted shift operator with the positive weight sequence $\{w_n\}_{n=1}^{\infty}$, and let $s\in\sigma_{sp}(T)\setminus\mathbb{R}$.  If $\boldsymbol{x}=(x_1,\ x_2,\ x_3,\ \cdots)$ is an $S$-eigenvector of $T$ with respect to $s$. Then, for any $n\geq 2$,
\[
x_{n+1}=
\frac{1}{w_1w_2\cdots w_n}\left(\frac{s^n-\bar{s}^n}{s-\bar{s}}\cdot w_1\cdot x_2-\frac{s^n\bar{s}-s\bar{s}^n}{s-\bar{s}}\cdot x_1\right).
\]
\end{lemma}
    \begin{proof}
    Since $\boldsymbol{x}=(x_1,\ x_2,\ \cdots,\ x_n)$ is an $S$-eigenvector of $T$ with respect to $s$, we have
    \[
    \begin{pmatrix}
|s|^2 & -2s_0w_1 & w_1w_2 & 0 & 0 & 0 &  \cdots \\
0 & |s|^2 & -2s_0w_2 & w_2w_3 & 0 & 0 & \cdots \\
0 & 0 & |s|^2 & -2s_0w_3 & w_3w_4 & 0 & \cdots \\
0 & 0 & 0 & |s|^2 & -2s_0w_4 & w_4w_5 & \cdots \\
\vdots & \vdots & \vdots & \vdots & \vdots & \vdots & \ddots
\end{pmatrix}
\cdot
\begin{pmatrix}
x_1 \\
x_2 \\
x_3 \\
x_4 \\
\vdots
\end{pmatrix}
=0,
\]
where $s_0=\rm{Re}(s)$. Then, one can see that for $n\geq 2$,
\begin{equation*}\label{eq9}
x_{n+1}-w_n^{-1}\bar{s}x_n=w_{n}^{-1}\cdots w_2
^{-1}s^{n-1}(x_2-\bar{s}w_1^{-1}x_1).
\end{equation*}
Consequently. we obtain that
\[
\begin{aligned}
x_{n+1}=&w_{n}^{-1}\cdots w_2
^{-1}\left((\sum\limits_{k=0}^{n-1}s^{n-k}\bar{s}^{k})\cdot x_2-(\sum\limits_{k=0}^{n-2}s^{n-k}\bar{s}^{k})\cdot|s|^2w_1^{-1}x_1 \right)\\
= & \frac{1}{w_1w_2\cdots w_n}\left(\frac{s^n-\bar{s}^n}{s-\bar{s}}\cdot w_1\cdot x_2-\frac{s^n\bar{s}-s\bar{s}^n}{s-\bar{s}}\cdot x_1\right).
\end{aligned}
\]
    \end{proof}

\begin{proposition}\label{SPR}
Let $T\in L_r(\mathcal{H}_q)$ be a backward unilateral weighted shift operator with the positive weight sequence $\{w_n\}_{n=1}^{\infty}$. Then,
the $S$-point spectrum radius of $T$ is
\[
r_{sp}(T)=\lim\limits_{n\to\infty}\inf |w_{1}w_{2}\cdots w_{n}|^{{1}/{n}}.
\]
Furthermore, let
\[
\Omega=\{s\in\mathbb{H}\setminus\mathbb{R};\ |s|<r_{sp}(T)\}.
\]
If $r_{sp}(T)>0$, then $T\in B_2^{s}(\Omega)$.
\end{proposition}

    \begin{proof}
    By Lemma \ref{BUSEV}, one can see that if $s\in\mathbb{H}$ satisfies
    \[
    |s|<\lim\limits_{n\to\infty}\inf |w_{1}w_{2}\cdots w_{n}|^{{1}/{n}},
    \]
    then $s\in\sigma_{sp}(T)$ and $\dim_{\mathbb{H}}(T^2-2\mathrm{Re}(s)T+|s|^2)=2$.

We claim that if $s\in\mathbb{H}$ satisfies
\[
|s|>\lim\limits_{n\to\infty}\inf |w_{1}w_{2}\cdots w_{n}|^{{1}/{n}},
\]
then $s\notin\sigma_{sp}(T)$. Suppose that $s\in\sigma_{sp}(T)$. By Theorem \ref{RSPE}, we also have $s\in\sigma_{r}(T)$. Notice that the vector $\boldsymbol{x}\in\ker(T-\boldsymbol{I}s)$ must have the following form
\[
\boldsymbol{x}=(q, \frac{q\cdot s}{w_1}, \cdots, \frac{q\cdot s^n}{w_1w_2\cdots w_n}, \cdots),
\]
where $q\in \mathbb{H}$. However, following from $|s|>\lim\limits_{n\to\infty}\inf |w_{1}w_{2}\cdots w_{n}|^{{1}/{n}}$, the vector $\boldsymbol{x}$ is not in $\mathcal{H}_q$ since its norm is infinity. This completes the proof of the claim.

Following from the above discussions, the $S$-point spectrum radius of $T$ is
\[
r_{sp}(T)=\lim\limits_{n\to\infty}\inf |w_{1}w_{2}\cdots w_{n}|^{{1}/{n}}.
\]

Now, suppose that $r_{sp}(T)>0$. For any $s\in\Omega_{red}$, consider the operator $T^2-2\mathrm{Re}(s)T+|s|^2$. It is not difficult to see that $(T^2-2\mathrm{Re}(s)T+|s|^2)_{\mathbb{C}}$ is a surjective on $(\mathcal{H}_q)_{\mathbb{C}}$. Then, $T^2-2\mathrm{Re}(s)T+|s|^2$ is a surjective on $\mathcal{H}_{q}$. Notice that
\[
\left({\bigvee}_{\mathbb{C}}\{\ker(T-\boldsymbol{I}s)_{\mathbb{C}}; \ s\in\Omega_{red}\}\right){\bigvee}_{\mathbb{C}}\left({\bigvee}_{\mathbb{C}}\{\ker(T-\boldsymbol{I}\bar{s})_{\mathbb{C}}; \ s\in\Omega_{red}\}\right)=(\mathcal{H}_q)_{\mathbb{C}}.
\]
By Theorem \ref{RSPE}, we  have
\[
{\bigvee}_{\mathbb{H}} \{\ker(T^2-2\mathrm{Re}(s) T+\vert s \vert^2),\ s\in\Omega_{red}\}=\mathcal{H}_q.
\]
Therefore, we obtain that $T\in B_2^{s}(\Omega) (=B_2^{s}(\Omega_{red}))$.
  \end{proof}

Let $\Omega_1$ and $\Omega_2$ be two distinct domain in the upper half plane. Let $m$ and $n$ be two distinct positive integer. It is possible that $B_m^{s}(\Omega_1)\bigcap B_n^{s}(\Omega_2)\neq\emptyset$.
We provide the following example, which comes from \cite{HT}, to illustrate this phenomenon.

\begin{example}\label{TCI}
	Let $T$ be the standard backward unilateral shift operator on $\mathcal{H}_q$. Let
	\[
	\Omega_1=\{z\in\mathbb{C};\ |z-\frac{i}{2}|<1,\ |z+\frac{i}{2}|>1,\ \mathrm{Im}(z)>0\}
	\]
	and
	\[
	\Omega_2=\{z\in\mathbb{C};\ |z-\frac{i}{2}|<1,\ |z+\frac{i}{2}|<1,\ \mathrm{Im}(z)>0\}.
	\]
	Then, $T+\frac{i}{2}I\in B_1^{sp}(\Omega_1)$ and $T+\frac{i}{2}I\in B_2^{sp}(\Omega_2)$.
	\end{example}
		
		\begin{proof}
		The matrix representation of the operator $T+\frac{i}{2}$ under the orthonormal basis  $\{e_n\}_{n=1}^{\infty}$ is
    $$
    \begin{pmatrix}
    \frac{i}{2} & 1 & 0 & 0 & \cdots \\
    0 & \frac{i}{2} & 1 & 0 & \cdots \\
    0 & 0 & \frac{i}{2} & 1 & \cdots \\
    \vdots & \vdots & \vdots & \vdots & \ddots
    \end{pmatrix}.
$$
     For any $\lambda\in\Omega_1$, write
    $$
    \boldsymbol{v}_1=\left(1,\ \lambda-\frac{i}{2},\ \left(\lambda-\frac{i}{2}\right)^2,\ \left(\lambda-\frac{i}{2}\right)^3,\ \cdots\right)^T.
    $$
   Then, one can see that
   \[
   \ker(T^2-2\mathrm{Re}(\lambda) T+\vert \lambda \vert^2)=\{\boldsymbol{v}_1q;\ q\in\mathbb{H}\},
   \]
 and hence $T\in B_1^{s}(\Omega_1)$.

     For any $\omega\in\Omega_2$, write
     \begin{align*}
    &\boldsymbol{v}_2=\left(1,\ \omega-\frac{i}{2},\ \left(\omega-\frac{i}{2}\right)^2,\ \left(\omega-\frac{i}{2}\right)^3,\ \cdots\right)^T, \\
    &\boldsymbol{v}_3=\left(1,\ \omega+\frac{i}{2},\ \left(\omega+\frac{i}{2}\right)^2,\ \left(\omega+\frac{i}{2}\right)^3,\ \cdots\right)^T.
    \end{align*}
Then, one can see that
\[
\ker(T^2-2\mathrm{Re}(\omega) T+\vert \omega \vert^2)=\{\boldsymbol{v}_2p+\boldsymbol{v}_3q;\ p,\ q\in\mathbb{H}\},
\]
and hence  $T\in B_2^{s}(\Omega_2)$.
		\end{proof}

		\section{Geometric rigidity for operators in $B_n^{s}(\Omega_q)$}
In this section, we will give a rigidity theorem for Hermitian right holomorphic quaternionic vector bundles. Furthermore, it is shown that two operators in $B_n^{s}(\Omega_q)$ are unitarily equivalent if and only if the associate bundles are equivalent as Hermitian right holomorphic quaternionic vector bundles.

\begin{definition}\label{CG}
	Let $\Omega$ be a domain in $\mathbb{C}$ and let $t_1,\ t_2:\ \Omega\rightarrow\mathcal{G}r(n,\ \mathcal{H}_q)$ be two right holomorphic curves. We
	say that $t_1$ and $t_2$ are congruent if there exists a quaternionic unitary operator $U$ on $\mathcal{H}_q$ such that $t_2=Ut_1$.
\end{definition}

\begin{definition}\label{QBE}
	Let $\Delta$ be a domain in $\mathbb{C}$ and let $t_1,\ t_2:\Delta\rightarrow\mathcal{G}r(n,\ \mathcal{H}_q)$ be two right holomorphic curves. We say that $E_{t_1}$ and $E_{t_2}$ are equivalent as Hermitian right holomorphic quaternionic vector bundles, if there exists a bundle map $\Phi: E_{t_1} \rightarrow E_{t_2}$ satisfying the following conditions.
	\begin{enumerate}
		\item[(1)] \ For any $\omega\in\Delta$, $\Phi\mid_{t_1(\omega)}:t_1(\omega)\rightarrow t_2(\omega)$ is a quaternionic right linear isomorphism. \\
		\item[(2)] \ For any right holomorphic cross-section $\gamma$ in $E_{t_1}$, $\Phi(\gamma)$ is a right holomorphic cross-section  in $E_{t_2}$.\\
		\item[(3)] \ For any right holomorphic cross-sections $\gamma$ and $\widetilde{\gamma}$ in $E_{t_1}$,
		\[
		\langle \gamma^{(m)},\ \widetilde{\gamma}\rangle_{\mathbb{H}}=\langle(\Phi(\gamma))^{(m)},\ \Phi(\widetilde{\gamma}) \rangle_{\mathbb{H}}.
		\]
	\end{enumerate}
Moreover, we call $\Phi$ an equivalent map from $E_{t_1}$ to $E_{t_2}$.
\end{definition}

The condition $(3)$ in the  above definition could be seemed as a strong version of isometry. It is worthy noticing that $\langle \gamma,\ \widetilde{\gamma}\rangle_{\mathbb{H}}=\langle\Phi(\gamma),\ \Phi(\widetilde{\gamma}) \rangle_{\mathbb{H}}$ does not imply $\langle \gamma^{(m)},\ \widetilde{\gamma}\rangle_{\mathbb{H}}=\langle(\Phi(\gamma))^{(m)},\ \Phi(\widetilde{\gamma}) \rangle_{\mathbb{H}}$, which is different from Hermitian holomorphic complex vector bundles (for instance in \cite{CD}). The essential reason is also the noncommutativity of quaternion multiplication. Let $f$ and $g$ be two real differentiable quaternionic-valued functions over a domain in the complex plane. The total differential of $f$ is as follows,
\[
	df=\frac{\partial f}{\partial x}dx+\frac{\partial f}{\partial y}dy
	= \frac{\partial f}{\partial z}dz+\frac{\partial f}{\partial \bar{z}}d\bar{z}.
\]
We always have
\[
d(fg)=(df)g+f(dg).
\]
However, if $f$ and $g$ are two quaternionic-valued right holomorphic functions, it happens in general that
\[
\frac{\partial (\overline{g}f)}{\partial z}dz\ne \overline{g}\frac{\partial f}{\partial z}dz,\ \ \ \ \ \
\frac{\partial (\overline{g}f)}{\partial \overline{z}}d\overline{z}\ne  \overline{(\frac{\partial g}{\partial z}dz)} f.
\]
Let $f(z)=f_1(z)+\boldsymbol{j}f_2(z)$ and $g(z)=g_1(z)+\boldsymbol{j}g_2(z)$ where $f_1$, $f_2$, $g_1$ and $g_2$ are complex-valued holomorphic functions. Then
    $$fg=(f_1g_1-\overline{f_2}g_2)+\boldsymbol{j}(f_2g_1+\overline{f_1}g_2),\ .$$

    However, when $f,\ g$ are quaternionic-valued right holomorphic functions, then
    $$\begin{aligned}
    \overline{g}f&=(\overline{g_1}f_1-\overline{g_2}f_2)+\boldsymbol{j}(g_1f_2-g_2f_1), \\
    \frac{\partial (\overline{g}f)}{\partial z}dz & =((\overline{g_1}f_1'+\overline{g_2}f'_2)+\boldsymbol{j}(g_1'f_2+g_1f_2'-g_2'f_1-g_2f_1'))dz, \\
    \frac{\partial (\overline{g}f)}{\partial \overline{z}}d\overline{z} & =(\overline{g'}f_1+\overline{g_2'}f_2)d\overline{z}, \\
    \overline{g}\frac{\partial f}{\partial z}dz & =((\overline{g_1}f_1'+\overline{g_2}f'_2)+\boldsymbol{j}(g_1f_2'-g_2f_1'))dz ,  \\
    \overline{(\frac{\partial g}{\partial z}dz)}
    f & =(\overline{g'}f_1+\overline{g_2'}f_2)d\overline{z}+\boldsymbol{j}(g_1'f_2
    -g_2'f_1)dz.
    \end{aligned}$$
    Consequently,
    $$
    d (\overline{g}f)=\frac{\partial (\overline{g}f)}{\partial z}dz+\frac{\partial (\overline{g}f)}{\partial \overline{z}}dz =
    \overline{g}\frac{\partial f}{\partial z}dz+\overline{(\frac{\partial g}{\partial z}dz)}f,
    $$
    but
    $$
    \frac{\partial \overline{g}f}{\partial z}dz\ne \overline{g}\frac{\partial f}{\partial z}dz,\ \ \ \ \ \ \
    \frac{\partial \overline{g}f}{\partial \overline{z}}dz\ne \overline{(\frac{\partial g}{\partial z}dz)}f.
    $$
Therefore, for any right holomorphic cross-sections $\gamma$ and $\widetilde{\gamma}$,
    \begin{equation}\label{TDiff}
    d\langle\gamma, \ \widetilde{\gamma}\rangle_{\mathbb{H}}=\frac{\partial}{\partial z}\langle\gamma, \ \widetilde{\gamma}\rangle_{\mathbb{H}}dz+\frac{\partial}{\partial\overline{z}}\langle\gamma, \ \widetilde{\gamma}\rangle_{\mathbb{H}}d\overline{z}=\langle\frac{\partial\gamma}{\partial z} dz, \ \widetilde{\gamma}\rangle_{\mathbb{H}}+\langle\gamma, \ \frac{\partial\widetilde{\gamma}}{\partial z}dz\rangle_{\mathbb{H}},
     \end{equation}
    but
    $$
\frac{\partial}{\partial z}\langle\gamma, \ \widetilde{\gamma}\rangle_{\mathbb{H}}dz\neq \langle\frac{\partial\gamma}{\partial z} dz, \ \widetilde{\gamma}\rangle_{\mathbb{H}}, \ \ \  \frac{\partial}{\partial\overline{z}}\langle\gamma, \ \widetilde{\gamma}\rangle_{\mathbb{H}}d\overline{z}\neq \langle\gamma, \ \frac{\partial\widetilde{\gamma}}{\partial z}dz\rangle_{\mathbb{H}}.
$$
So,  $\langle \gamma,\ \widetilde{\gamma}\rangle_{\mathbb{H}}=\langle\Phi(\gamma),\ \Phi(\widetilde{\gamma}) \rangle_{\mathbb{H}}$ does not imply $\langle \gamma^{(m)},\ \widetilde{\gamma}\rangle_{\mathbb{H}}=\langle(\Phi(\gamma))^{(m)},\ \Phi(\widetilde{\gamma}) \rangle_{\mathbb{H}}$.

\begin{theorem}[Rigidity]\label{RT}
	Let $\Delta$ be a domain in $\mathbb{C}$ and let $t_1,\ t_2:\Delta\rightarrow\mathcal{G}r(n,\ \mathcal{H}_q)$ be two right holomorphic curves satisfying $\bigvee\limits_{\mathbb{H}}\{ t_1(\omega),\ \omega\in\Delta\}=\bigvee\limits_{\mathbb{H}}\{ t_2(\omega),\ \omega\in\Delta\}=\mathcal{H}_q$.
    Then $t_1$ and $t_2$ are congruent if and only if $E_{t_1}$ and $E_{t_2}$ are equivalent as Hermitian right holomorphic quaternionic vector bundles over $\Delta$.
\end{theorem}
	\begin{proof}
	Suppose that $t_1$ and $t_2$ are congruent. Let $U$ be a unitary operator such that $Ut_1=t_2$. Obviously, $U$ is just an equivalent map from $E_{t_1}$ to $E_{t_2}$.
	
	Now, suppose that $\Phi:E_{t_1}\rightarrow E_{t_2}$ is an equivalent map. Let $\{\gamma_i\}_{i=1}^{n}$ be a right holomorphic frame of $E_{t_1}$ over $\Delta$. Write $\widetilde{\gamma_i}=\Phi(\gamma_i)$ for $i=1,\cdots,n$. Then, $\{\widetilde{\gamma_i}\}_{i=1}^{n}$ is a right holomorphic frame of $E_{t_2}$ over $\Delta$.

   We claim that for any $m,k\in\mathbb{N}$ and any $i,j=1,\cdots,n$,
\begin{equation}\label{eq1}
	\langle\gamma_i^{(m)},\ \gamma_j^{(k)}\rangle_{\mathbb{H}}\ =\langle\widetilde{\gamma_i}^{(m)},\ \widetilde{\gamma_i}^{(k)}\rangle_{\mathbb{H}}.
\end{equation}

Following from Definition \ref{QBE}, one can see that the equation (\ref{eq1}) holds for $k=0$ and all $m\in\mathbb{N}$. Assume that the equation (\ref{eq1}) holds for $k=0,1,\cdots, l$ and all $m\in\mathbb{N}$.

Given any $m\in\mathbb{N}$. Following from the equation (\ref{TDiff}), we have
    \[
    \begin{aligned}
	d\langle\gamma_i^{(m)}, \ {\gamma_j}^{(l)}\rangle_{\mathbb{H}}=&\langle\gamma_i^{(m+1)}dz, \ {\gamma_j}^{(l)}\rangle_{\mathbb{H}}+\langle\gamma_i^{(m)}, \ {\gamma_j}^{(l+1)}dz\rangle_{\mathbb{H}} \\
	=&\langle\gamma_i^{(m+1)}, \ {\gamma_j}^{(l)}\rangle_{\mathbb{H}}dz+d\bar{z}\langle\gamma_i^{(m)}, \ {\gamma_j}^{(l+1)}\rangle_{\mathbb{H}}.
	\end{aligned}
\]
Similarly,
    \[
	d\langle\widetilde{\gamma_i}^{(m)}, \ \widetilde{\gamma_j}^{(l)}\rangle_{\mathbb{H}}=\langle\widetilde{\gamma_i}^{(m+1)}, \ \widetilde{\gamma_j}^{(l)}\rangle_{\mathbb{H}}dz+d\bar{z}\langle\widetilde{\gamma_i}^{(m)}, \ \widetilde{\gamma_j}^{(l+1)}\rangle_{\mathbb{H}}.
\]
By the assumption, we have known that
\[
d\langle\gamma_i^{(m)}, \ {\gamma_j}^{(l)}\rangle_{\mathbb{H}}=d\langle\widetilde{\gamma_i}^{(m)}, \ \widetilde{\gamma_j}^{(l)}\rangle_{\mathbb{H}} \ \ \ \text{and} \ \ \ \langle\gamma_i^{(m+1)}dz, \ {\gamma_j}^{(l)}\rangle_{\mathbb{H}}=\langle\widetilde{\gamma_i}^{(m+1)}, \ \widetilde{\gamma_j}^{(l)}\rangle_{\mathbb{H}}dz.
\]
Then
    \[
\begin{aligned}
	d\bar{z}\langle\gamma_i^{(m)}, \ {\gamma_j}^{(l+1)}\rangle_{\mathbb{H}}=&d\langle\gamma_i^{(m)}, \ {\gamma_j}^{(l)}\rangle_{\mathbb{H}}-\langle\gamma_i^{(m+1)}, \ {\gamma_j}^{(l)}\rangle_{\mathbb{H}}dz \\
	=& 	d\langle\widetilde{\gamma_i}^{(m)}, \ \widetilde{\gamma_j}^{(l)}\rangle_{\mathbb{H}}-\langle\widetilde{\gamma_i}^{(m+1)}, \ \widetilde{\gamma_j}^{(l)}\rangle_{\mathbb{H}}dz \\
	=& d\bar{z}\langle\widetilde{\gamma_i}^{(m)}, \ \widetilde{\gamma_j}^{(l+1)}\rangle_{\mathbb{H}}.
\end{aligned}
\]
This proves the claim.

Given any $\omega_0\in\Delta$. Notice that
\[
{\bigvee}_{\mathbb{H}}\{ t_1(\omega);\ \omega\in\Delta\}={\bigvee}_{\mathbb{H}}\{ t_2(\omega);\ \omega\in\Delta\}=\mathcal{H}_q
\]
implies
\[
\begin{aligned}
& {\bigvee}_{\mathbb{H}}\{\gamma_i^{(m)}(\omega_0);\ i=1,\cdots,n \ \text{and} \ m\in\mathbb{N}\} \\
=&{\bigvee}_{\mathbb{H}}\ \widetilde{\gamma_i}^{(m)}(\omega_0);\ i=1,\cdots,n \ \text{and} \ m\in\mathbb{N}\} \\
=&\mathcal{H}_q.
\end{aligned}
\]
We define an operator $U_{\omega}:\mathcal{H}_q\rightarrow \mathcal{H}_q$ by
\[
U\gamma_i^{(m)}(\omega_0)=\widetilde{\gamma_i}^{(m)},
\]
for all $i=1,\cdots,n$ and $m\in\mathbb{N}$.

Since for all $i, j=1,\cdots,n$ and $m,k\in\mathbb{N}$,
    \[
    \langle U_{\omega}\gamma_i^{(m)}(\omega),\ U_{\omega}\gamma_j^{(k)}(\omega)\rangle=\langle\widetilde{\gamma_i}^{(m)}, \ \widetilde{\gamma_j}^{(k)}\rangle_{\mathbb{H}}=\langle\gamma_i^{(m)}(\omega),\ \gamma_j^{(k)}(\omega)\rangle,
    \]
 we obtain that $U$ is a surjective isometry. Moreover,  for any $\omega$ in a neighborhood of $\omega_0$, we have
    \[
    \begin{aligned}
     U(\gamma_i(\omega))
     = & U(\sum_{m=0}^{\infty}\gamma_i^{(m)}(\omega_0)\frac{(\omega-\omega_0)^{m}}{m!}) \\
     = & \sum_{m=0}^{\infty}\widetilde{\gamma_i}^{(m)}(\omega_0)\frac{(\omega-\omega_0)^{m}}{m!} \\
     = & \widetilde{\gamma_i}(\omega),
    \end{aligned}
\]
for every $i=1,\cdots,n$. Therefore, $U$ is a unitary operator such that $t_2=Ut_1$. The proof is finished.
    \end{proof}

    \begin{theorem}\label{MT}
	Let $T_1,\ T_2\in B_n^{s}(\Omega_q)$. Then, $T_1$ and $T_2$ are quaternion unitarily equivalent if and only if $E_{T_1}$ and $E_{T_2}$ are equivalent as Hermitian right holomorphic quaternionic vector bundles.
	\end{theorem}
    \begin{proof}
    Recall that $t_i(\omega)=\ker(T_i^2-2\mathrm{Re}(\omega)T_i+|\omega|^2)$, $i=1, 2$, for any $\omega\in\Omega_{red}$ and $E_{T_i}$ is the pull-back right holomorphic quaternionic vector bundles. Suppose that $U$ is a  quaternionic unitary operator $U$ on $\mathcal{H}_q$ such that
    $UT_1=T_2U$. For any $\boldsymbol{x}\in\ker(T_1^2-2\mathrm{Re}(\omega)T_1+|\omega|^2)$, one can see that $U\boldsymbol{x}\in\ker(T_2^2-2\mathrm{Re}(\omega)T_2+|\omega|^2)$ since
    $$
     (T_2^2-2\mathrm{Re}(\omega)T_2+|\omega|^2)(U\boldsymbol{x})
    =  U(T_1^2-2\mathrm{Re}(\omega)T_1+|\omega|^2)\boldsymbol{x}
    =  0.
    $$
    Then, we have $Ut_1=t_2$. Therefore, $t_1$ and $t_2$ are congruent and consequently, $E_{T_1}$ and $E_{T_2}$ are equivalent as Hermitian right holomorphic quaternionic vector bundles.

    On the other hand, suppose that $E_{T_1}$ and $E_{T_2}$ are equivalent as Hermitian right holomorphic quaternionic vector bundles. Then, by Theorem \ref{RT}, the right holomorphic curves $t_1$ and $t_2$ are congruent. Let $V$ be a unitary operator on $\mathcal{H}_q$ such that $Vt_1=t_2$. Let
    $\{\gamma_j\}_{j=1}^n$ is a right holomorphic frame of $E_{T_1}$ over a domain $\Delta\subseteq\Omega_{red}$.
    By Lemma \ref{orthonormal basis P}, $\gamma_j(\omega)\in \ker(T_1-\boldsymbol{I}\omega)$ and $V(\gamma_j(\omega))\in \ker(T_2-\boldsymbol{I}\omega)$ for any $\omega\in\Delta$.
    Then, for every $j=1, \cdots, n$,
       $$
       VT_1(\gamma_{j}(\omega))=V(\gamma_{j}(\omega)\cdot \omega)=V(\gamma_{j}(\omega))\cdot \omega=T_2V(\gamma_{j}(\omega)).
       $$
    Since
    \[
    \bigvee\limits_{\mathbb{H}}\{ \gamma_j(\omega); \ j=1, \cdots, n \ \text{and} \ \omega\in\Delta\}=\mathcal{H}_q,
    \]
    we have $VT_1=T_2V$. Therefore, $T_1$ and $T_2$ are quaternion unitarily equivalent. The proof is finished.
    \end{proof}

\section{Operators in $B_1^{s}(\Omega)$}

In this section, we focus on the operators in $B_1^{s}(\Omega)$. Let $T\in B_1^{s}(\Omega_q)$. Given any $\omega_0\in\Omega_{red}$. Let $\gamma$ be a right holomorphic frame (a right holomorphic cross-section that is non-trivial everywhere) of $E_T$ over a neighborhood $\Delta$ of $\omega_0$. We could obtain an orthonormal basis $\{e_i(\omega_0)\}_{i=0}^{\infty}$ from $\{\gamma^{(i)}(\omega_0)\}_{i=0}^{\infty}$ by Gram-Schmidt orthogonalization, that is,
\begin{equation}\label{orthonormal basis }
	e_i(\omega_0)=\begin{cases}
		\frac{\gamma(\omega_0)}{\Vert\gamma(\omega_0)\Vert}\ & i=0 \\
		\frac{\gamma^{(i)}(\omega_0)-\sum_{k=1}^{i-1}e_k(\omega_0){\langle\gamma^{(i)}(\omega_0),\ e_k(\omega_0)\rangle_{\mathbb{H}}}}
		{\|\gamma^{(i)}(\omega_0)-\sum_{k=1}^{i-1}e_k(\omega_0){\langle\gamma^{(i)}(\omega_0),\ e_k(\omega_0)\rangle_{\mathbb{H}}}\|}\ & i\ge1,\\
	\end{cases}
\end{equation}
For any $i\in\mathbb{N}$, let
\[
\mathcal{H}_i(\omega_0)\triangleq{\bigvee}_{\mathbb{H}}\{e_k(\omega_0);\ k=0,\cdots,i\}={\bigvee}_{\mathbb{H}}\{\gamma^{(k)}(\omega_0); \ k=0,\cdots,i\}.
\]
and $P_i:\ \mathcal{H}_{q}\mapsto\mathcal{H}_{i}(\omega_0)$ be the orthogonal projection. Then
\[
e_i(\omega_0)=\frac{\gamma^{(i)}(\omega_0)-P_{i-1}(\gamma^{(i)}(\omega_0))}
{\|\gamma^{(i)}(\omega_0)-P_{i-1}(\gamma^{(i)}(\omega_0))\|}.
\]
By Lemma \ref{orthonormal basis P}, we have $Te_0=e_0\omega_0$ and for $i\in\mathbb{N}$,
\[
\begin{aligned}
Te_{i+1}(\omega_0)=& \ T\left(\frac{\gamma^{(i+1)}(\omega_0)-P_{i}(\gamma^{(i+1)}(\omega_0))}
{\|\gamma^{(i+1)}(\omega_0)-P_{i}(\gamma^{(i+1)}(\omega_0))\|}\right) \\
=& \ \frac{T\gamma^{(i+1)}(\omega_0)-TP_{i}(\gamma^{(i)}(\omega_0))}
{\|\gamma^{(i+1)}(\omega_0)-P_{i}(\gamma^{(i+1)}(\omega_0))\|} \\
=& \ \frac{\gamma^{(i+1)}(\omega_0)\cdot\omega_0+k\gamma^{(i+1)}(\omega_0)-P_{i}TP_i(\gamma^{(i+1)}(\omega_0))}
{\|\gamma^{(i+1)}(\omega_0)-P_{i}(\gamma^{(i+1)}(\omega_0))\|} \\
=& \ \frac{\gamma^{(i+1)}(\omega_0)\cdot\omega_0-P_{i}(\gamma^{(i+1)}(\omega_0))\cdot\omega_0}
{\|\gamma^{(i+1)}(\omega_0)-P_{i}(\gamma^{(i+1)}(\omega_0))\|}+ \\
&\ \frac{P_{i}(\gamma^{(i+1)}(\omega_0))\cdot\omega_0+k\gamma^{(i)}(\omega_0)-P_{i}TP_i(\gamma^{(i+1)}(\omega_0))}
{\|\gamma^{(i+1)}(\omega_0)-P_{i}(\gamma^{(i+1)}(\omega_0))\|} \\
=&\ e_{i+1}(\omega_0)\cdot\omega_0+P_i\left(\frac{\gamma^{(i+1)}(\omega_0)\cdot\omega_0+k\gamma^{(i)}(\omega_0)-TP_i(\gamma^{(i+1)}(\omega_0))}{\|\gamma^{(i+1)}(\omega_0)-P_{i}(\gamma^{(i+1)}(\omega_0))\|}\right)\\
=&\ e_{i+1}(\omega_0)\cdot\omega_0+P_i(Te_{i+1}(\omega_0)).
\end{aligned}
\]
Then, under the orthonormal basis $\{e_i(\omega_0)\}_{i=0}^{\infty}$, $T$ has the matrix representation $N(\omega_0)$ as follows,
\[
   N(\omega_0)= \begin{pmatrix}
	\omega & \langle Te_1(\omega_0),\ e_0(\omega_0)\rangle_{\mathbb{H}} & \langle Te_2(\omega_0),\ e_0(\omega_0)\rangle_{\mathbb{H}} & \langle Te_3(\omega_0),\ e_0(\omega_0)\rangle_{\mathbb{H}} & \cdots \\
	0 & \omega & \langle Te_2(\omega_0),\ e_1(\omega_0)\rangle_{\mathbb{H}} & \langle Te_2(\omega_0),\ e_1(\omega_0)\rangle_{\mathbb{H}} & \cdots \\
	0 & 0 & \omega & \langle Te_3(\omega_0),\ e_2(\omega_0)\rangle_{\mathbb{H}} & \cdots \\
	0 &	0 & 0 & \omega & \cdots \\
	\vdots & \vdots & \vdots & \vdots & \ddots
\end{pmatrix}.
\]
We call $N(\omega_0)$ the canonical matrix representation of $T$ at $\omega_0$. Next, we will study the canonical matrix representation.

    \begin{lemma}\label{FR}
 Let $T\in B_1^{s}(\Omega_q)$. Suppose that $\gamma$ and $\widetilde{\gamma}$ are two right holomorphic  frames  of $E_T$ over a domain $\Delta$. Then, there exists a complex-valued holomorphic function $f$ on $\Delta$ such that $\widetilde{\gamma}(\omega)=\gamma(\omega)f(\omega)$ for any $\omega\in\Delta$.
    \end{lemma}
    \begin{proof}
   Since $\gamma(\omega),\widetilde{\gamma}(\omega)\in\ker(T-\boldsymbol{I}\omega)$ and $\textrm{dim}_{\mathbb{C}}\ker(T-\boldsymbol{I}\omega)=1$  for any $\omega\in\Delta$, there exists a complex-valued holomorphic function $f$ on $\Delta$ such that $\widetilde{\gamma}(\omega)=\gamma(\omega)f(\omega)$. Notice that
   \[
   0=\frac{\partial \widetilde{\gamma}}{\partial\overline{\omega}}=\frac{\partial \gamma}{\partial\overline{\omega}}\cdot f+ \gamma\cdot\frac{\partial f}{\partial\overline{\omega}}=\gamma\cdot\frac{\partial f}{\partial\overline{\omega}}.
   \]
    Then. $\frac{\partial f}{\partial\overline{\omega}}=0$ for any $\omega\in\Delta$, and hence $f$ is a complex-valued holomorphic function on $\Delta$ such that $\widetilde{\gamma}(\omega)=\gamma(\omega)f(\omega)$ for any $\omega\in\Delta$.
    \end{proof}

\begin{lemma}\label{Canonical}
Let $T\in B_1^{s}(\Omega_q)$. Given any $\omega_0\in\Omega_{red}$. Suppose that $\gamma$ and $\widetilde{\gamma}$ are two right holomorphic  frames  of $E_T$ over a neighborhood $\Delta$ of $\omega_0$. Let $\{e_i(\omega_0)\}_{i=0}^{\infty}$ and $\{\widetilde{e_i}(\omega_0)\}_{i=0}^{\infty}$ be the orthonormal basis induced by $\{\gamma^{(i)}(\omega_0)\}_{i=0}^{\infty}$ and $\{\widetilde{\gamma}^{(i)}(\omega_0)\}_{i=0}^{\infty}$, respectively. Furthermore, write $N(\omega_0)$ and  $\widetilde{N}(\omega_0)$ the matrix representations of $T$ under the orthonormal basis $\{e_i(\omega_0)\}_{i=0}^{\infty}$ and $\{\widetilde{\gamma}^{(i)}(\omega_0)\}_{i=0}^{\infty}$, respectively. Then, there exists $\theta_0\in[0,\ 2\pi)$ such that for any $i\in\mathbb{N}$,
\[
\widetilde{e_i}(\omega_0)=e_i(\omega_0)\cdot\mathrm{e}^{\boldsymbol{i}\theta_0}.
\]
Moreover,
\[
\widetilde{N}(\omega_0)=\mathrm{e}^{-\boldsymbol{i}\theta_0}\cdot N(\omega_0) \cdot\mathrm{e}^{\boldsymbol{i}\theta_0}.
\]
\end{lemma}
    \begin{proof}
    By	Lemma \ref{FR}, there exists a complex-valued holomorphic function $f$ on $\Delta$ such that $\widetilde{\gamma}(\omega)=\gamma(\omega)f(\omega)$ for any $\omega\in\Delta$. Since for any $i\in\mathbb{N}$,
    $$
    (\gamma(\omega)f(\omega))^{(i+1)}=\gamma^{(i+1)}(\omega)f(\omega)+\sum_{k=0}^{i}
    \begin{pmatrix}
    i+1 \\
    k
    \end{pmatrix}\gamma^{(k)}(\omega)f^{(i+1-k)}(\omega),
$$
and
\[
\mathcal{H}_i(\omega_0)\triangleq{\bigvee}_{\mathbb{H}}\{e_k(\omega_0);\ k=0,\cdots,i\}={\bigvee}_{\mathbb{H}}\{\gamma^{(k)}(\omega_0); \ k=0,\cdots,i\}.
\]
    We have $\widetilde{e_{0}}(\omega_0)=e_{0}\cdot\frac{f(\omega_0)}{\|f(\omega_0)\|}$ and
    $$
    \begin{aligned}
    \widetilde{e_{i+1}}(\omega_0) & =\frac{(\gamma f)^{(i+1)}(\omega_0)-P_i((\gamma f)^{(i+1)}(\omega_0))}
    {\|(\gamma f)^{(i+1)}(\omega_0)-P_i((\gamma f)^{(i+1)}(\omega_0))\|} \\
    & =\frac{\gamma^{(i+1)}(\omega_0)f(\omega_0)-P_i(\gamma^{(i+1)}(\omega_0)f(\omega_0))}{\|\gamma^{(i+1)}(\omega_0)f(\omega_0)-P_i(\gamma^{(i+1)}(\omega_0)f(\omega_0))\|} \\
    & =\frac{(\gamma^{(i+1)}(\omega_0)-P_i(\gamma^{(i+1)}(\omega_0)))f(\omega_0)}{\|(\gamma^{(i+1)}(\omega_0)-P_i(\gamma^{(i+1)}(\omega_0)))f(\omega_0)]|} \\
    & =e_{i+1}\cdot\frac{f(\omega_0)}{\|f(\omega_0)\|}.
    \end{aligned}
    $$
    Let $\theta_0\in[0,\ 2\pi)$ such that $\mathrm{e}^{\boldsymbol{i}\theta_0}=\frac{f(\omega_0)}{\|f(\omega_0)\|}$. Then, $\widetilde{e_i}(\omega_0)=e_i(\omega_0)\cdot\mathrm{e}^{\boldsymbol{i}\theta_0}$ for any $i\in\mathbb{N}$.

    Furthermore, one can see that for any $i,j\in\mathbb{N}$ with $i\geq j$,
    \[
    \begin{aligned}
    \langle T\widetilde{e_i}(\omega_0),\ \widetilde{e_j}(\omega_0)\rangle_{\mathbb{H}}=&\langle Te_i(\omega_0)\cdot\mathrm{e}^{\boldsymbol{i}\theta_0},\ e_j(\omega_0)\cdot\mathrm{e}^{\boldsymbol{i}\theta_0}\rangle_{\mathbb{H}} \\
    =&\mathrm{e}^{-\boldsymbol{i}\theta_0}\cdot \langle Te_i(\omega_0),\ e_j(\omega_0)\rangle_{\mathbb{H}} \cdot\mathrm{e}^{\boldsymbol{i}\theta_0}.
    \end{aligned}
\]
Therefore, we obtain that
\[
\widetilde{N}(\omega_0)=\mathrm{e}^{-\boldsymbol{i}\theta_0}\cdot N(\omega_0) \cdot\mathrm{e}^{\boldsymbol{i}\theta_0}.
\]
    \end{proof}
\begin{remark}
	For $T\in B_1^{s}(\Omega_q)$ and $\omega_0\in\Omega_{red}$,  Lemma \ref{Canonical} shows that the matrix representation $N(\omega_0)$ is unique in the sense of modulus $\mathrm{ad}_{\mathrm{e}^{\boldsymbol{i}\theta}}$ actions, where $\mathrm{ad}_{\mathrm{e}^{\boldsymbol{i}\theta}}(N(\omega_0)):=\mathrm{e}^{-\boldsymbol{i}\theta}\cdot N(\omega_0) \cdot\mathrm{e}^{\boldsymbol{i}\theta}$ for $\theta\in\mathbb{R}$. So, it is reasonable to name $N(\omega_0)$ the canonical matrix representation of $T$ at $\omega_0$.
\end{remark}

    \begin{theorem}\label{CbyC}
Let $T_1,\ T_2\in B_1^{s}(\Omega_q)$.  Given any $\omega_0\in\Omega_{red}$. Let $N_1(\omega_0)$ and $N_2(\omega_0)$ be the canonical matrix representations of $T_1$ and $T_2$ at $\omega_0$, respectively.
Then, $T_1$ and $T_2$ are quaternion unitarily equivalent if and only if
there exists $\theta\in[0,\ 2\pi)$ such that $N_2(\omega_0)=\mathrm{e}^{-\boldsymbol{i}\theta_0}\cdot N_1(\omega_0) \cdot\mathrm{e}^{\boldsymbol{i}\theta_0}$.
\end{theorem}
    \begin{proof}
 Since $N_1(\omega_0)$ and $N_2(\omega_0)$ are the canonical matrix representations of $T_1$ and $T_2$ at $\omega_0$ respectively, it is obviously that $N_2(\omega_0)=\mathrm{e}^{-\boldsymbol{i}\theta_0}\cdot N_1(\omega_0) \cdot\mathrm{e}^{\boldsymbol{i}\theta_0}$ implies $T_1$ and $T_2$ are quaternion unitarily equivalent.

 Now, suppose that there exists a quaternion unitary operator $U$ such that $UT_1=T_2U$. Let $\gamma_1$ be a right holomorphic fame of $E_{T_1}$ and $\{e_i^1(\omega_0)\}_{i=0}^{\infty}$ be the orthonormal basis induced by $\{\gamma_1^{(i)}(\omega_0)\}_{i=0}^{\infty}$. Then, by Theorem \ref{MT}, $\gamma_2:=U\gamma_1$ is a right holomorphic fame of $E_{T_2}$. Let $\{e_i^2(\omega_0)\}_{i=0}^{\infty}$ be the orthonormal basis induced by $\{\gamma_2^{(i)}(\omega_0)\}_{i=0}^{\infty}$. One can see that $Ue_i^1(\omega_0)=e_i^2(\omega_0)$ for every $i\in\mathbb{N}$. Therefore, by Lemma \ref{Canonical}, there exists $\theta\in[0,\ 2\pi)$ such that $N_2(\omega_0)=\mathrm{e}^{-\boldsymbol{i}\theta_0}\cdot N_1(\omega_0) \cdot\mathrm{e}^{\boldsymbol{i}\theta_0}$.
  \end{proof}

In the classical Cowen-Douglas (complex) operator theory, curvature is a complete invariant for the unitary classification of $B_1(\Omega)$ (see \cite{CD}). Review that if $T\in B_1(\Omega)$ is a complex linear operator and $\gamma$ is a non-vanishing holomorphic cross-section of $E_{T}$, then the curvature $K_T$ of the Hermitian holomorphic complex vector bundle $E_T$ is the real analytic function defined by
$$
\mathcal{K}_T(\omega)=-\frac{\partial^2 \log \|\gamma(\omega)\|^2}{\partial\omega\partial\overline{\omega}}=\frac{|\langle\gamma'(\omega),\ \gamma(\omega)\rangle|^2-\|\gamma'(\omega)\|^2\|\gamma(\omega)\|^2}{\|\gamma(\omega)\|^4},
$$
for any $\omega\in\Omega$.

Similarly, we also could define the curvature of a $1$-dimensional Hermitian right holomorphic quaternionic vector bundle. Let $T\in B_1(\Omega_q)$ and $\gamma$ be any non-vanishing right holomorphic cross-section of $E_{T}$. The curvature $\mathcal{K}_T$ of $E_T$ is defined by
$$
\mathcal{K}_T(\omega)=-\frac{\partial^2 \log \|\gamma(\omega)\|^2}{\partial\omega\partial\overline{\omega}}.
$$
Notice that the curvature $\mathcal{K}_T$ of $E_T$ equals to the curvature of $(E_T)_{\mathbb{C}}$ (the complex representation bundle of $E_T$), i.e.,
$$
\mathcal{K}_T(\omega)=-\frac{\partial^2 \log \|(\gamma(\omega))_{\mathbb{C}}\|^2}{\partial\omega\partial\overline{\omega}}=\frac{|\langle(\gamma'(\omega))_{\mathbb{C}},\ (\gamma(\omega))_{\mathbb{C}}\rangle|^2-\|(\gamma'(\omega))_{\mathbb{C}}\|^2\|(\gamma(\omega))_{\mathbb{C}}\|^2}{\|(\gamma(\omega))_{\mathbb{C}}\|^4}.
$$
However, in general,
$$
\mathcal{K}_T(\omega)\neq\frac{|\langle\gamma'(\omega),\ \gamma(\omega)\rangle_{\mathbb{H}}|^2-\|\gamma'(\omega)\|^2\|\gamma(\omega)\|^2}{\|\gamma(\omega)\|^4}.
$$
\begin{example}\label{CNDU}
Fix an orthonormal basis $\{\mathbf{e}_n\}_{n=0}^{\infty}$ in $\mathcal{H}_q$. Let $T$ and $\widetilde{T}$ be two linear operators on $\mathcal{H}_q$ with the following matrix representations under the orthonormal basis $\{\mathbf{e}_n\}_{n=0}^{\infty}$,
\[
T=
\begin{pmatrix}
\boldsymbol{i} & 1+2\boldsymbol{ji} & 1+\boldsymbol{j} & 0 & 0 & \cdots \\
0 & \boldsymbol{i} & 1 & 0 & 0 & \cdots \\
0 & 0 & \boldsymbol{i} & 1 & 0 & \cdots \\
0 & 0 & 0 & \boldsymbol{i} & 1  & \cdots \\
0 & 0 & 0 & 0 & \boldsymbol{i}  & \cdots \\
\vdots & \vdots & \vdots & \vdots & \vdots & \ddots \\
\end{pmatrix}
\]
and
\[
\widetilde{T}=\begin{pmatrix}
	\boldsymbol{i} & \frac{1}{2}-\frac{1}{2}\boldsymbol{j} & 1 & 0 & 0 & \cdots \\
	0 & \boldsymbol{ji} & 1+\boldsymbol{j} & 0 & 0 & \cdots \\
	0 & 0 & \boldsymbol{i} & 1 & 0 & \cdots \\
	0 & 0 & 0 & \boldsymbol{i} & 1  & \cdots \\
	0 & 0 & 0 & 0 & \boldsymbol{i}  & \cdots \\
	\vdots & \vdots & \vdots & \vdots & \vdots & \ddots \\
\end{pmatrix}.
\]
Let $\Omega_{red}=\{\omega\in\mathbb{C};\ |\omega-i|<1\}$. It is not difficult to see that $T,\widetilde{T}\in B_1^s(\Omega_{red})$. We will show that $T$ and $\widetilde{T}$ have the same curvature, but they are not quaternion unitarily equivalent.

   For any $\omega\in\Omega_{red}$, let
    \[
    \gamma(\omega)  =(1+(1+\boldsymbol{j})(\omega-\boldsymbol{i}),\ (\omega-\boldsymbol{i}),\ (\omega-\boldsymbol{i})^2,\
    (\omega-\boldsymbol{i})^3,\ \cdots)^t
    \]
    and
    \[
    \widetilde{\gamma}(\omega)=(1+\omega-\boldsymbol{i},\ (1+\boldsymbol{j})(\omega-\boldsymbol{i}),\ (\omega-\boldsymbol{i})^2,\
    (\omega-\boldsymbol{i})^3,\ \cdots)^t.
   \]
    Then, $\gamma$ and $\widetilde{\gamma}$ are  right holomorphic frames of $E_T$ and
    $E_{\widetilde{T}}$, respectively.  For any $\omega\in\Omega_{red}$, one can see that
    \[
    \|\gamma(\omega)\|^2=|1+{\omega}-\boldsymbol{i}|^2+2|\omega-\boldsymbol{i}|^2
    +\sum_{n=2}^{\infty}|\omega-\boldsymbol{i}|^{2n}=\|\widetilde{\gamma}(\omega)\|^2.
    \]
    Consequently, $T$ and $\widetilde{T}$ have the same curvature since
       \[
       \mathcal{K}_T(\omega)=-\frac{\partial^2 \log \|\gamma(\omega)\|^2}{\partial{\omega}\partial{\overline{\omega}}}=-\frac{\partial^2 \log \|\widetilde{\gamma}(\omega)\|^2}{\partial{\omega}\partial{\overline{\omega}}}=\mathcal{K}_{\widetilde{T}}(\omega).
       \]

    Now, consider the canonical matrix representations of $T$ and $\widetilde{T}$ at $\boldsymbol{i}$, denoted by $N_T(\boldsymbol{i})$ and $N_{\widetilde{T}}(\boldsymbol{i})$ respectively.
    Let $\{e_i(\boldsymbol{i})\}_{i=0}^{\infty}$ and $\{\widetilde{e_i}(\boldsymbol{i})\}_{i=0}^{\infty}$ be the orthonormal basis induced by $\{\gamma^{(i)}(\boldsymbol{i})\}_{i=0}^{\infty}$ and $\{\widetilde{\gamma}^{(i)}(\boldsymbol{i})\}_{i=0}^{\infty}$, respectively.
    Then,
    $$\begin{aligned}
    \gamma(\boldsymbol{i}) & =(1,\ 0,\ 0,\ 0,\ \cdots),\ \gamma'(\boldsymbol{i}) =(1+\boldsymbol{j},\ 1, \ 0,\ 0,\ \cdots), \\
    \widetilde{\gamma}(\boldsymbol{i}) & =(1,\ 0,\ 0,\ 0,\ \cdots),\ \widetilde{\gamma}'(\boldsymbol{i}) =(1,\ 1+\boldsymbol{j}, \ 0,\ 0,\ \cdots),
    \end{aligned}$$
    and furthermore,
     $$\begin{aligned}
    	e_0(\boldsymbol{i}) & =(1,\ 0,\ 0,\ 0,\ \cdots),\ e_1(\boldsymbol{i}) =(0,\ 1, \ 0,\ 0,\ \cdots), \\
    	\widetilde{e_0}(\boldsymbol{i}) & =(1,\ 0,\ 0,\ 0,\ \cdots),\ \widetilde{e_1}(\boldsymbol{i}) =(0,\ \frac{\sqrt{2}}{2}(1+\boldsymbol{j}), \ 0,\ 0,\ \cdots).
    \end{aligned}$$
    We have
     $$\begin{aligned}
     &\langle(Te_1(\boldsymbol{i}),\ e_0(\boldsymbol{i}) \rangle_{\mathbb{H}}=1+2\boldsymbol{ji}, \\
     & \langle(T\widetilde{e_1}(\boldsymbol{i}),\ \widetilde{e_0}(\boldsymbol{i}) \rangle_{\mathbb{H}}=\frac{\sqrt{2}}{2}.
    \end{aligned}$$
    Then, $\langle(Te_1(\boldsymbol{i}),\ e_0(\boldsymbol{i}) \rangle_{\mathbb{H}}$ and $\langle(T\widetilde{e_1}(\boldsymbol{i}),\ \widetilde{e_0}(\boldsymbol{i}) \rangle_{\mathbb{H}}$ are not axially symmetric and consequently, $\mathrm{e}^{-\boldsymbol{i}\theta}N_{T}(\boldsymbol{i})\mathrm{e}^{\boldsymbol{i}\theta}\ne N_{\widetilde{T}}(\boldsymbol{i})$ for every $\theta\in[0,\ 2\pi)$. Therefore, by Theorem \ref{CbyC},
    $T$ and $\widetilde{T}$ are not quaternion unitarily equivalent.
\end{example}

\begin{remark}
	The above specific example shows that the operators in $B_1^{s}(\Omega_{q})$ with the same curvature does not imply they are quaternion unitarily equivalent.
	Example {6.15} in \cite{HT} aims to do the same thing. However, the two operators in Example {6.15} in \cite{HT} are quaternion unitarily equivalent. So, Example {6.15} in \cite{HT} should be replaced by the above example.
\end{remark}

At the end of this article, let us show that two operators in $B_1^{s}(\Omega_{q})$ are quaternion unitarily equivalent if and only if their complex representation operators are unitarily equivalent.

\begin{theorem}
Let $T,\widetilde{T}\in B_1^{s}(\Omega_{q})$. Then,
$T$ and $\widetilde{T}$ are quaternion unitarily equivalent if and only if
$T_{\mathbb{C}}$ and $\widetilde{T}_{\mathbb{C}}$ are unitarily equivalent as bounded linear operators acting on $({\mathcal{H}_q})_\mathbb{C}$.
\end{theorem}

    \begin{proof}
    Suppose that $U$ is a quaternionic unitary operator such that
    $\widetilde{T}=UTU^*$. Notice that $(U^*)_{\mathbb{C}}=(U)_{\mathbb{C}}^*$ and $(U)_{\mathbb{C}}$ is a unitary operator. Then, it follows from $(\widetilde{T})_{\mathbb{C}}=(U)_{\mathbb{C}}(T)_{\mathbb{C}}(U)_{\mathbb{C}}^*$  that $T_{\mathbb{C}}$ and $\widetilde{T}_{\mathbb{C}}$ are unitarily equivalent as bounded linear operators acting on $({\mathcal{H}_q})_\mathbb{C}$.

    Now, suppose that there is a unitary operator $W$ acting on $({\mathcal{H}_q})_\mathbb{C}$ such that $\widetilde{T}_{\mathbb{C}}=WT_{\mathbb{C}}W^*$. Given any $\omega_0\in\Omega_{red}$. Let
    $$
    \gamma(\omega)=\alpha(\omega)+\boldsymbol{j}\beta(\omega)=
    (\alpha_{1}(\omega),\ \alpha_{2}(\omega),\ \cdots)+\boldsymbol{j}(\beta_{1}(\omega),\ \beta_{2}(\omega),\cdots)
    $$
    be a non-vanishing right holomorphic cross-section of $E_T$ over a neighborhood $\Delta$ of $\omega_0$, where $\alpha(\omega)$ and $\beta(\omega)$ are complex vectors. For any $\omega\in\Delta$, one can see that
    $$
    \boldsymbol{0}=((T-\boldsymbol{I}\omega)\gamma(\omega))_{\mathbb{C}}=(T_{\mathbb{C}}-\omega)
    \begin{pmatrix}
    	\alpha(\omega) \\
    	\beta(\omega)
    	\end{pmatrix}.
    $$
    Denote
    $$
    W
    \begin{pmatrix}
    	\alpha(\omega) \\
    	\beta(\omega)
    \end{pmatrix} \\
    =\begin{pmatrix}
    	\widetilde{\alpha}(\omega) \\
    	\widetilde{\beta}(\omega)
    \end{pmatrix}.
$$
and let
    $$
    \widetilde{\gamma}(\omega)=\widetilde{\alpha}(\omega)+\boldsymbol{j}\widetilde{\beta}(\omega).
    $$
   Since
    $$\begin{aligned}
    (\widetilde{T}-\boldsymbol{I}\omega)_{\mathbb{C}}
    \begin{pmatrix}
    \widetilde{\alpha}(\omega) \\
    \widetilde{\beta}(\omega)
    \end{pmatrix}
    = & (\widetilde{T}_{{\mathbb{C}}}-\omega)W
    \begin{pmatrix}
    \alpha(\omega) \\
    \beta(\omega)
    \end{pmatrix}\\
    = & W(T_{{\mathbb{C}}}-\omega)
    \begin{pmatrix}
    \alpha(\omega) \\
    \beta(\omega)
    \end{pmatrix}\\
    =&\boldsymbol{0}
    \end{aligned},$$
    we have
    $(\widetilde{\gamma}(\omega))_{\mathbb{C}}\in\ker{(\widetilde{T}-\boldsymbol{I}\omega)_{\mathbb{C}}}$ and consequently, $\widetilde{\gamma}(\omega)\in\ker{(\widetilde{T}-\boldsymbol{I}\omega)}$.
Notice that
\[
(\gamma(\omega)\boldsymbol{j})_{\mathbb{C}}=\begin{pmatrix}
	-\overline{\beta(\omega)} \\
	\overline{\alpha(\omega)}
\end{pmatrix} \ \ \ \text{and} \ \ \ \ (\widetilde{\gamma}(\omega)\boldsymbol{j})_{\mathbb{C}}=\begin{pmatrix}
-\overline{\widetilde{\beta}(\omega)} \\
\overline{\widetilde{\alpha}(\omega)}
\end{pmatrix}.
\]
Since
\[
({T}-\boldsymbol{I}\overline{\omega})(\gamma(\omega)\boldsymbol{j})=\gamma(\omega)\omega\boldsymbol{j}-\gamma(\omega)\boldsymbol{j}\overline{\omega}=0,
\]
we have
    $$
    \begin{aligned}
  (\widetilde{T}_{{\mathbb{C}}}-\overline{\omega})W\begin{pmatrix}
-\overline{\beta(\omega)} \\
\overline{\alpha(\omega)}
\end{pmatrix}
	= & W({T}_{\mathbb{C}}-\boldsymbol{I}\overline{\omega})\begin{pmatrix}
		-\overline{\beta(\omega)} \\
		\overline{\alpha(\omega)}
	\end{pmatrix} \\
	= & W({T}-\overline{\omega})_{\mathbb{C}}(\gamma(\omega)\boldsymbol{j})_{\mathbb{C}} \\
	= & W(({T}-\boldsymbol{I}\overline{\omega})(\gamma(\omega)\boldsymbol{j}))_{\mathbb{C}} \\
	= & \boldsymbol{0}.
\end{aligned}
$$
Then, $\left( W\begin{pmatrix}
	-\overline{\beta(\omega)} \\
	\overline{\alpha(\omega)}
\end{pmatrix} \right)_q\in\ker(\widetilde{T}-\boldsymbol{I}\overline{\omega})$. Since
\[
(\widetilde{T}-\boldsymbol{I}\overline{\omega})(\widetilde{\gamma}(\omega)\boldsymbol{j})=0,
\]
it follows from Lemma \ref{FR} that there is a complex holomorphic function $f$ on $\Delta$ such that
    $$
    W\begin{pmatrix}
    	-\overline{\beta(\omega)} \\
    	\overline{\alpha(\omega)}
    \end{pmatrix}
    =\begin{pmatrix}
    	-\overline{\widetilde{\beta}(\omega)f(\omega)} \\
    	\overline{\widetilde{\alpha}(\omega)f(\omega)}
    \end{pmatrix}.
$$
    Since $W$ is a unitary operator, one can see that
    \[
     \left\|\begin{pmatrix}
     	-\overline{\beta(\omega)} \\
     	\overline{\alpha(\omega)}
     \end{pmatrix} \right\|=\left\|W\begin{pmatrix}
     -\overline{\beta(\omega)} \\
     \overline{\alpha(\omega)}
 \end{pmatrix} \right\|=\left\|\begin{pmatrix}
 -\overline{\widetilde{\beta}(\omega)f(\omega)} \\
 \overline{\widetilde{\alpha}(\omega)f(\omega)}
\end{pmatrix} \right\|=\left\|\begin{pmatrix}
-\overline{\beta(\omega)} \\
\overline{\alpha(\omega)}
\end{pmatrix} \right\|\cdot|f(\omega)|,
    \]
  and consequently for every $\omega\in\Delta$,
    \[
    |f(\omega)|=1.
    \]
   Then, there exists $\theta_0\in[0,\ 2\pi)$ such that
    $f(\omega)=\mathrm{e}^{\boldsymbol{i}\theta_0}$ for every $\omega\in\Delta$. Define the linear operator $V:(\mathcal{H}_q)_{\mathbb{C}}\rightarrow(\mathcal{H}_q)_{\mathbb{C}}$ by
    $$
    V
    \begin{pmatrix}
    \alpha^{(l)}(\omega_0) \\
    \beta^{(l)}(\omega_0) \\
    \end{pmatrix}
    =\begin{pmatrix}
    \widetilde{\alpha}^{(l)}(\omega_0)\mathrm{e}^{-\boldsymbol{i}\theta_0/2} \\
    \widetilde{\beta}^{(l)}(\omega_0)\mathrm{e}^{-\boldsymbol{i}\theta_0/2} \\
    \end{pmatrix},\
    V
    \begin{pmatrix}
    -\overline{\beta^{(l)}(\omega_0)} \\
    \overline{\alpha^{(l)}(\omega_0)} \\
    \end{pmatrix}
    =\begin{pmatrix}
    -\overline{\widetilde{\beta}^{(l)}(\omega_0)}\mathrm{e}^{\boldsymbol{i}\theta_0/2} \\
    \overline{\widetilde{\alpha}^{(l)}(\omega_0)}\mathrm{e}^{\boldsymbol{i}\theta_0/2} \\
    \end{pmatrix}
    $$
    for any $l\in\mathbb{N}$. It is not difficult to see that $VT_{\mathbb{C}}=\widetilde{T}_{\mathbb{C}}V$.

  For every $k,\ l\in\mathbb{N}$, we have
     \begin{equation*}
 	\begin{aligned}
 		\langle V
 		\begin{pmatrix}
 			\alpha^{(l)}(\omega_0) \\
 			\beta^{(l)}(\omega_0)
 		\end{pmatrix},\
 		V
 		\begin{pmatrix}
 			\alpha^{(k)}(\omega_0) \\
 			\beta^{(k)}(\omega_0)
 		\end{pmatrix}
 		\rangle
 		= & \langle
 		\begin{pmatrix}
 			\widetilde{\alpha}^{(l)}(\omega_0)\mathrm{e}^{-\boldsymbol{i}\theta_0/2} \\
 			\widetilde{\beta}^{(l)}(\omega_0)\mathrm{e}^{-\boldsymbol{i}\theta_0/2} \\
 		\end{pmatrix},\
 		\begin{pmatrix}
 			\widetilde{\alpha}^{(k)}(\omega_0)\mathrm{e}^{-\boldsymbol{i}\theta_0/2} \\
 			\widetilde{\beta}^{(k)}(\omega_0)\mathrm{e}^{-\boldsymbol{i}\theta_0/2} \\
 		\end{pmatrix}
 		\rangle \\
 		= &
 		\langle
 		\begin{pmatrix}
 			\widetilde{\alpha}^{(l)}(\omega_0) \\
 			\widetilde{\beta}^{(l)}(\omega_0)\\
 		\end{pmatrix},\
 		\begin{pmatrix}
 			\widetilde{\alpha}^{(k)}(\omega_0) \\
 			\widetilde{\beta}^{(k)}(\omega_0) \\
 		\end{pmatrix}
 		\rangle \\
 		= & \langle
 		W\begin{pmatrix}
 			\alpha^{(l)}(\omega_0) \\
 			\beta^{(l)}(\omega_0)
 		\end{pmatrix},\
 		W\begin{pmatrix}
 			\alpha^{(k)}(\omega_0) \\
 			\beta^{(k)}(\omega_0)
 		\end{pmatrix}
 		\rangle \\
 		= & \langle
 		\begin{pmatrix}
 			\alpha^{(l)}(\omega_0) \\
 			\beta^{(l)}(\omega_0)
 		\end{pmatrix},\
 		\begin{pmatrix}
 			\alpha^{(k)}(\omega_0) \\
 			\beta^{(k)}(\omega_0)
 		\end{pmatrix}
 		\rangle,
 	\end{aligned}
 \end{equation*}

    \begin{equation*}
    \begin{aligned}
     \langle V
    \begin{pmatrix}
    -\overline{\beta^{(l)}(\omega_0)} \\
    \overline{\alpha^{(l)}(\omega_0)}
    \end{pmatrix},\
    V
   \begin{pmatrix}
   	-\overline{\beta^{(k)}(\omega_0)} \\
   	\overline{\alpha^{(k)}(\omega_0)}
   \end{pmatrix}
    \rangle
    = &      \langle
    \begin{pmatrix}
    	-\overline{\widetilde{\beta}^{(l)}(\omega_0)}\mathrm{e}^{\boldsymbol{i}\theta_0/2} \\
    	\overline{\widetilde{\alpha}^{(l)}(\omega_0)}\mathrm{e}^{\boldsymbol{i}\theta_0/2}
    \end{pmatrix},\
      \begin{pmatrix}
     	-\overline{\widetilde{\beta}^{(k)}(\omega_0)}\mathrm{e}^{\boldsymbol{i}\theta_0/2} \\
     	\overline{\widetilde{\alpha}^{(k)}(\omega_0)}\mathrm{e}^{\boldsymbol{i}\theta_0/2}
     \end{pmatrix}
    \rangle \\
    = &
    \langle
    \begin{pmatrix}
    	-\overline{\widetilde{\beta}^{(l)}(\omega_0)}\mathrm{e}^{\boldsymbol{i}\theta_0} \\
    	\overline{\widetilde{\alpha}^{(l)}(\omega_0)}\mathrm{e}^{\boldsymbol{i}\theta_0}
    \end{pmatrix},\
    \begin{pmatrix}
    	-\overline{\widetilde{\beta}^{(k)}(\omega_0)}\mathrm{e}^{\boldsymbol{i}\theta_0} \\
    	\overline{\widetilde{\alpha}^{(k)}(\omega_0)}\mathrm{e}^{\boldsymbol{i}\theta_0}
    \end{pmatrix}
    \rangle \\
    = & \langle W
    \begin{pmatrix}
    	-\overline{\beta^{(l)}(\omega_0)} \\
    	\overline{\alpha^{(l)}(\omega_0)}
    \end{pmatrix},\
    W
    \begin{pmatrix}
    	-\overline{\beta^{(k)}(\omega_0)} \\
    	\overline{\alpha^{(k)}(\omega_0)}
    \end{pmatrix}
    \rangle \\
     = & \langle
     \begin{pmatrix}
     	\alpha^{(l)}(\omega_0) \\
     	\beta^{(l)}(\omega_0)
     \end{pmatrix},\
       \begin{pmatrix}
     	\alpha^{(k)}(\omega_0) \\
     	\beta^{(k)}(\omega_0)
     \end{pmatrix}
     \rangle,
    \end{aligned}
    \end{equation*}
and
     \begin{equation*}
	\begin{aligned}
		\langle V
		\begin{pmatrix}
			\alpha^{(l)}(\omega_0) \\
			\beta^{(l)}(\omega_0)
		\end{pmatrix},\
		 V
		\begin{pmatrix}
			-\overline{\beta^{(k)}(\omega_0)} \\
			\overline{\alpha^{(k)}(\omega_0)}
		\end{pmatrix}
		\rangle
		= & \langle
		\begin{pmatrix}
			\widetilde{\alpha}^{(l)}(\omega_0)\mathrm{e}^{-\boldsymbol{i}\theta_0/2} \\
			\widetilde{\beta}^{(l)}(\omega_0)\mathrm{e}^{-\boldsymbol{i}\theta_0/2} \\
		\end{pmatrix},\
		\begin{pmatrix}
			-\overline{\widetilde{\beta}^{(k)}(\omega_0)}\mathrm{e}^{\boldsymbol{i}\theta_0/2} \\
			\overline{\widetilde{\alpha}^{(k)}(\omega_0)}\mathrm{e}^{\boldsymbol{i}\theta_0/2}
		\end{pmatrix}
		\rangle \\
		= &
		\langle
		\begin{pmatrix}
			\widetilde{\alpha}^{(l)}(\omega_0) \\
			\widetilde{\beta}^{(l)}(\omega_0)\\
		\end{pmatrix},\
		\begin{pmatrix}
			-\overline{\widetilde{\beta}^{(k)}(\omega_0)}\mathrm{e}^{\boldsymbol{i}\theta_0} \\
			\overline{\widetilde{\alpha}^{(k)}(\omega_0)}\mathrm{e}^{\boldsymbol{i}\theta_0}
		\end{pmatrix}
		\rangle \\
		= & \langle
		W\begin{pmatrix}
			\alpha^{(l)}(\omega_0) \\
			\beta^{(l)}(\omega_0)
		\end{pmatrix},\
		W\begin{pmatrix}
			-\overline{\beta^{(k)}(\omega_0)} \\
			\overline{\alpha^{(k)}(\omega_0)}
		\end{pmatrix}
		\rangle \\
		= & \langle
		\begin{pmatrix}
			\alpha^{(l)}(\omega_0) \\
			\beta^{(l)}(\omega_0)
		\end{pmatrix},\
		\begin{pmatrix}
			-\overline{\beta^{(k)}(\omega_0)} \\
			\overline{\alpha^{(k)}(\omega_0)}
		\end{pmatrix}
		\rangle,
	\end{aligned}
\end{equation*}

   Then, $V$ is an isometry. Following from
   \[
   {\bigvee}_{\mathbb{C}}\{\widetilde{\gamma}^{(i)}(\omega_0); i\in\mathbb{N}\ \}=\mathcal{H}_q,
   \]
   $V$ is surjective ahd hence $V$ is a unitary operator satisfying $\widetilde{T}_{\mathbb{C}}=VT_{\mathbb{C}}V^*$.

Define the linear operator $U:\mathcal{H}_q\rightarrow\mathcal{H}_q$ by for any $\omega\in\Delta$,
\[
U(\gamma(\omega))=\widetilde{\gamma}(\omega)\cdot \mathrm{e}^{-\boldsymbol{i}\theta_0/2}.
\]
Notice that for any $l\in\mathbb{N}$,
\[
U_{\mathbb{C}}
\begin{pmatrix}
	\alpha^{(l)}(\omega_0) \\
	\beta^{(l)}(\omega_0) \\
\end{pmatrix}
=(U\gamma^{(l)}(\omega_0))_{\mathbb{C}}=(\widetilde{\gamma}^{(l)}(\omega_0)\mathrm{e}^{-\boldsymbol{i}\theta_0/2})_{\mathbb{C}}=\begin{pmatrix}
	\widetilde{\alpha}^{(l)}(\omega_0)\mathrm{e}^{-\boldsymbol{i}\theta_0/2} \\
	\widetilde{\beta}^{(l)}(\omega_0)\mathrm{e}^{-\boldsymbol{i}\theta_0/2} \\
\end{pmatrix},
\]
and
\[
U_{\mathbb{C}}
\begin{pmatrix}
	-\overline{\beta^{(l)}(\omega_0)} \\
	\overline{\alpha^{(l)}(\omega_0)} \\
\end{pmatrix}
=(U\gamma^{(l)}(\omega_0)\boldsymbol{j})_{\mathbb{C}}=(\widetilde{\gamma}^{(l)}(\omega_0)\mathrm{e}^{-\boldsymbol{i}\theta_0/2}\boldsymbol{j})_{\mathbb{C}}=\begin{pmatrix}
	-\overline{\widetilde{\beta}^{(l)}(\omega_0)}\mathrm{e}^{\boldsymbol{i}\theta_0/2} \\
	\overline{\widetilde{\alpha}^{(l)}(\omega_0)}\mathrm{e}^{\boldsymbol{i}\theta_0/2} \\
\end{pmatrix}.
\]
Then, $U_{\mathbb{C}}=V$ and hence $U$ is a quaternionic unitary operator with $\widetilde{T}=UTU^*$. This finishes the proof.
\end{proof}

\section*{Statements and Declarations}

\begin{itemize}
	\item Funding 
	
	The second author was partially supported by National Natural Science Foundation of China (Grant No. 12471120). The third author was partially supported by National Natural Science Foundation of China (Grant No. 12371129).
	\item Conflict of interest/Competing interests 
	
	The authors declare that there is no conflict of interest or competing interest.
	\item Ethics approval and consent to participate
	
	Not applicable.
	\item Author contribution 
	
	All authors contributed equally to this work.
\end{itemize}

\noindent



\begin{thebibliography}{10}

\bibitem{AF}
S.~L. Adler and D.~R. Finkelstein.
\newblock Quaternionic quantum mechanics and quantum fields.
\newblock {\em Physics Today}, 49(6):58--60, 1995.

\bibitem{CCKS}
P.~Cerejeiras, F.~Colombo, U.~Khler, and I.~Sabadini.
\newblock Perturbation of normal quaternionic operators.
\newblock {\em American Mathematical Society (AMS)}, 372(5):3257--3281, 2019.

\bibitem{CGSS}
F.~Colombo, G.~Gentili, I.~Sabadini, and D.~C. Struppa.
\newblock Extension results for slice regular functions of a quaternionic
  variable.
\newblock {\em Advances in Mathematics}, 222(5):1793--1808, 2009.

\bibitem{CSS}
F.~Colombo, I.~Sabadini, and D.~C. Struppa.
\newblock Noncommutative functional calculus.
\newblock {\em Springer Basel}, 2011.

\bibitem{CO}
J.~B. Conway.
\newblock A course in functional analysis.
\newblock {\em Springer-Verlag}, 1990.

\bibitem{CD}
M.~J. Cowen and R.~G. Douglas.
\newblock Complex geometry and operator theory.
\newblock {\em Acta Mathematica}, 141(1):187--261, 1978.

\bibitem{E}
G.~Emch.
\newblock M{\'e}canique quantique quaternionienne et relativit{\'e} restreinte.
  i.
\newblock {\em Helvetica Physica Acta}, 36:739--769, 1969.

\bibitem{FJS}
D.~Finkelstein, J.~M. Jauch, and D.~Speiser.
\newblock Quaternionic representations of compact groups.
\newblock {\em Journal of Mathematical Physics}, 4(1):136--140, 1963.

\bibitem{FJSS}
D.~Finkelstein, M.~Jauch, S.~Schiminovich, and D.~Speiser.
\newblock Foundations of quaternion quantum mechanics.
\newblock {\em Journal of Mathematical Physics}, 3(2):207--220, 1962.

\bibitem{GS}
G.~Gentili and D.~C. Struppa.
\newblock A new theory of regular functions of a quaternionic variable.
\newblock {\em Advances in Mathematics}, 216(1):279--301, 2007.

\bibitem{GP}
R.~Ghiloni and A.~Perotti.
\newblock Slice regular functions on real alternative algebras.
\newblock {\em Advances in Mathematics}, 226(2):1662--1691, 2010.

\bibitem{HB}
L.~P. Horwitz and L.~C. Biedenharn.
\newblock Quaternion quantum mechanics: Second quantization and gauge fields.
\newblock {\em Annals of Physics}, 157(2):432--488, 1984.

\bibitem{HT}
B.~Hou and G.~Tian.
\newblock Geometry and operator theory on quaternionic {H}ilbert spaces.
\newblock {\em Annals of Functional Analysis}, 6(4):226--246, 2015.

\bibitem{HJJX}
Y.~Hou, K.~Ji, S.~Ji, and J.~Xu.
\newblock Geometry of holomorphic vector bundles and similarity of commuting
  tuples of operators.
\newblock {\em Journal of Operator Theory}, 91(1):169--202, 2024.

\bibitem{JJKX}
K.~Ji, S.~Ji, H-K. K, and J.~Xu.
\newblock The {C}owen-{D}ouglas theory for operator tuples and similarity.
\newblock {\em Complex Analysis and Operator Theory}, 19(2):No. 24, 43 pages,
  2025.

\bibitem{JGJ}
C.~Jiang, X.~Guo, and K.~Ji.
\newblock K-group and similarity classification of operators.
\newblock {\em Journal of Functional Analysis}, 225(1):167--192, 2005.

\bibitem{JJ}
C.~Jiang and K.~Ji.
\newblock Similarity classification of holomorphic curves.
\newblock {\em Advances in Mathematics}, 215(2):446--468, 2007.

\bibitem{JJK}
C.~Jiang, K.~Ji, and D.~K. Keshari.
\newblock Geometric similarity invariants of {C}owen-{D}ouglas operators, 2023.

\bibitem{JJW}
C.~Jiang, K.~Ji, and J.~Wu.
\newblock Similarity invariants of essentially normal {C}owen-{D}ouglas
  operators and chern polynomials.
\newblock {\em Israel Journal of Mathematics}, 248:229--270, 2022.

\bibitem{L}
H.~C. Lee.
\newblock Eigenvalues and canonical forms of matrices with quaternion
  coefficients.
\newblock {\em Proceedings of the Royal Irish Academy}, 52:253--260, 1949.

\bibitem{LS}
S.~D. Leo and G.~Scolarici.
\newblock Right eigenvalue equation in quaternionic quantum mechanics.
\newblock {\em Journal of Physics A General Physics}, 33(15):2971--2995(25),
  2000.

\bibitem{LSS}
S.~D. Leo, G.~Scolarici, and L.~Solombrino.
\newblock Quaternionic eigenvalue problem.
\newblock {\em Journal of Mathematical Physics}, 43(11):5815--5829, 2002.

\bibitem{M}
G.~Misra.
\newblock Operators in the {C}owen-{D}ouglas class and related topics.
\newblock In {\em Handbook of analytic operator theory}, CRC Press/Chapman Hall
  Handb. Math. Ser., pages 87--137. CRC Press, Boca Raton, FL, 2019.

\bibitem{NV}
S.~Natarajan and K.~Viswanath.
\newblock Quaternionic representations of compact metric groups.
\newblock {\em Journal of Mathematical Physics}, 8(3):582--589, 1967.

\bibitem{RK}
G.~Ramesh and P.~S. Kumar.
\newblock Spectral theorem for quaternionic normal operators: Multiplication
  form.
\newblock {\em Bulletin des Sciences Math{\'e}matiques}, 159:102840, 2020.

\bibitem{SC}
C.~S. Sharma and T.~J. Coulson.
\newblock Spectral theory for unitary operators on a quaternionic {H}ilbert
  space.
\newblock {\em Journal of Mathematical Physics}, 1987.

\bibitem{V}
K.~Viswanath.
\newblock Normal operations on quaternionic {H}ilbert spaces.
\newblock {\em Transactions of the American Mathematical Society},
  162:337--337, 1971.

\bibitem{W}
L.~A. Wolf.
\newblock Similarity of matrices in which the elements are real quaternions.
\newblock {\em Bulletin of the American Mathematical Society}, 42:737--743,
  1936.

\bibitem{Z}
F.~Zhang.
\newblock Quaternions and matrices of quaternions.
\newblock {\em Linear Algebra and Its Applications}, 251(2):21--57, 1997.

\end{thebibliography}
\end{document}